\documentclass[12pt]{article}

\usepackage{amsmath}
\usepackage{amssymb}
\usepackage{amsthm}
\usepackage{amsfonts}
\usepackage{latexsym}
\usepackage{cite}
\usepackage{psfrag}
\usepackage{mathrsfs}
\usepackage{url}
\usepackage{color}

\usepackage[colorlinks=true]{hyperref}
\hypersetup{urlcolor=blue, citecolor=red}

\makeatletter
\@addtoreset{equation}{section}
\makeatother

\newtheorem{theorem}{Theorem}[section]
\newtheorem{lemma}[theorem]{Lemma}

\newtheorem{proposition}[theorem]{Proposition}
\theoremstyle{definition}
\newtheorem{definition}[theorem]{Definition}
\newtheorem{notation}[theorem]{Notation}
\newtheorem{remark}[theorem]{Remark}
\newtheorem{example}[theorem]{Example}

\def\L{\mathcal{L}}
\def\O{\mathcal{O}}

\def\TT{\mathcal{T}}

\def\C{\mathscr{C}}

\def\M{\mathscr{M}}
\def\N{\mathscr{N}}

\def\Os{\mathscr{O}}

\def\PG{\mathrm{PG}}
\def\RC{\mathrm{RC}}
\def\Tr{\mathrm{T}}
\def\IC{\mathrm{IC}}
\def\EnG{\mathrm{En}\Gamma}
\def\TO{\mathrm{TO}}

\def\J{\mathbf{J}}
\def\MM{\mathbf{M}}
\def\Pf{\mathbf{P}}

\def\F{\mathbb{F}}
\newcommand{\Pb}{\mathbb{P}}
\newcommand{\Lb}{\mathbb{L}}

\def\Nb{\mathbb{N}}
\def\V{\mathbb{V}}

\newcommand{\A}{\mathfrak{A}}

\newcommand{\Mk}{\mathfrak{M}}
\newcommand{\Nk}{\mathfrak{N}}
\newcommand{\Pk}{\mathfrak{P}}
\newcommand{\Tk}{\mathfrak{T}}
\newcommand{\pk}{\mathfrak{p}}
\newcommand{\nk}{\mathfrak{n}}

\def\T{\text}
\def\db{\displaybreak[3]}
\def\dbn{\displaybreak[3]\notag}
\def\nt{\notag}

\begin{document}
\title
{Incidence matrices for the class $\mathcal{O}_6$ of lines external to the twisted cubic in $\mathrm{PG}(3,q)$
\date{}
\thanks{The research of S. Marcugini and F. Pambianco was supported in part by the Italian
National Group for Algebraic and Geometric Structures and their Applications (GNSAGA -
INDAM) (Contract No. U-UFMBAZ-2019-000160, 11.02.2019) and by University of Perugia
(Project No. 98751: Strutture Geometriche, Combinatoria e loro Applicazioni, Base Research
Fund 2017-2019).}
}
\maketitle
\begin{center}
{\sc Alexander A. Davydov}\\
 \emph{E-mail address:} alexander.davydov121@gmail.com\medskip\\
 {\sc Stefano Marcugini and
 Fernanda Pambianco }\\
 {\sc\small Department of  Mathematics  and Computer Science,  Perugia University,}\\
 {\sc\small Perugia, 06123, Italy}\\
 \emph{E-mail address:} \{stefano.marcugini, fernanda.pambianco\}@unipg.it
\end{center}

\textbf{Abstract.}We consider the structures of the plane-line and point-line incidence matrices of the projective space $\mathrm{PG}(3,q)$ connected with orbits of planes, points, and lines under the stabilizer group of the twisted cubic. In the literature, lines are partitioned into classes, each of which is a union of line orbits. In this paper, for all $q$, even and odd, we determine  the incidence matrices connected with a family of  orbits of the class named $\mathcal{O}_6$. This class contains lines external to the twisted cubic. The considered family include an essential part of all $\mathcal{O}_6$ orbits, whose complete classification is an open problem.



\section{Introduction}
Let $\F_{q}$ be the Galois field with $q$ elements, $\F_{q}^*=\F_{q}\setminus\{0\}$, $\F_q^+=\F_q\cup\{\infty\}$.

In the $N$-dimensional projective space $\PG(N,q)$ over $\F_q$, an $n$-arc  is a
set of $n$ points such that no $N +1$ points belong to the same hyperplane.

In $\PG(N,q)$ a normal rational curve is any $(q+1)$-arc projectively equivalent to the point set
$\{(t^N,t^{N-1},\ldots,t^2,t,1)|t\in \F_q^+\}$,  named  \emph{twisted cubic} for $N=3$; see \cite{Hirs_PG3q}.

The twisted cubic has many interesting properties and is connected with distinct combinatorial and applied problems, which led this curve to be widely studied, see for instance \cite{BonPolvTwCub,BrHirsTwCub,CapKorchSon,CasseGlynn82,CasseGlynn84,CosHirsStTwCub,%
GiulVincTwCub,Hirs_PG3q,KorchLanzSon,ZanZuan2010} and the references therein.

Of particular interest is to consider the action of the stabilizer group $G_q\cong \mathrm{PGL}(2,q)$ of the twisted cubic on points, lines, planes of PG(3,q) and the determination of their related incidence matrices. The investigations, based on the known classification of the point and plane orbits of $G_q$ given in \cite{Hirs_PG3q}, were started by D. Bartoli and the present authors in 2020 \cite{BDMP-TwCub}: the point-plane incidence matrix is determined and applied in coding theory  obtaining optimal multiple covering codes.  The results in \cite{BDMP-TwCub} are also useful to classify the cosets of the $[ q+1, q-3,5]_q 3$ generalized doubly-extended Reed-Solomon code of codimension $4$ through their weight distributions \cite{DMP_CosetsRScod4}.

For the study of plane-line and point-line incidence matrices a description of line orbits under $G_q$ is useful.
In \cite{Hirs_PG3q}, a partition of the lines in PG(3,q) into classes is given, each of which is a union of line orbits under $G_q$.
In \cite{DMP_OrbLineArX, DMP_OrbLineArXII, DMP_OrbLineMedit}  we have determined both the number and the structure of the orbits forming those unions, apart from one class denoted by $\mathcal{O}_6$, containing  lines external to  the twisted cubic that are not its chords or axes and do not lie in its osculating planes. Basing on the results of \cite{DMP_OrbLineArX, DMP_OrbLineArXII, DMP_OrbLineMedit}, in \cite{DMP_PlLineIncarX,DMP_PointLineInc,DMP_PlLineIncJG} the
plane-line and point-line incidence matrices connected with orbits of all line classes apart from $\mathcal{O}_6$ were obtained. For $\mathcal{O}_6$ some average and cumulative values are calculated.

The above results have attracted attention, and motivated some investigation using different techniques in \cite{BlokPelSzoDesi,GulLavFFA} and in the recent paper \cite{CePa}. In particular, in \cite{BlokPelSzoDesi} (for all $q \ge23$) and in \cite{GulLavFFA} (for finite fields of characteristic $>3$) line orbits (without those from $\mathcal{O}_6$) are obtained by other
methods than in \cite{DMP_OrbLineArX, DMP_OrbLineArXII, DMP_OrbLineMedit}. Also, in \cite{GulLavFFA} incidence matrices connected with the described
line orbits are given.
In \cite{CePa},  for even $q=2^n$, $n\ge3$
the $(q+1)$-arc $\mathcal{A}=\{(1,t,t^{2^h},t^{2^h+1})|t\in\F_q^+\} \subset\PG(3,q)$ with $gcd(n,h)=1$ (it is the twisted cubic for $h=1$), has been considered; it is shown that the orbits of points and of planes
under the projective stabilizer $G_h$ of $\mathcal{A}$ are similar to those under $G_1$ described in \cite{Hirs_PG3q}; moreover, the point-plane incidence matrix with
respect to $G_h$-orbits  mirrors the case h=1 described in  \cite{BDMP-TwCub}. In \cite{CePa}, it is also proved that for even $q$,  $q\equiv\xi\pmod3$, $\xi \in \{1,-1\}$, $G_h$ has $2q+7+\xi$  orbits on lines, providing a proof of a conjecture of ours \cite[Conjecture 8.2]{DMP_OrbLineArX,DMP_OrbLineMedit} in the case even $q$. Also, in \cite{CePa}, the point-line incidence matrix for even $q$ is given.

The class $\mathcal{O}_6$ is the most complicated since the number of orbits it contains depends on $q$,
whereas the remaining seven classes contain at most three orbits.

Recently, for all even and odd $q$, we have determined a family $\mathscr{F}$ of  orbits of this class $\mathcal{O}_6$; see \cite{DMP_OrbitsO6arX}. The orbits of the family depend on a parameter running over $\F_q^*$. Also, there is once more  special orbit with another description. The family $\mathscr{F}$ includes an essential part of all $\mathcal{O}_6$ orbits, whose complete classification remains an open problem.

 In this paper, using the properties of the family $\mathscr{F}$ from \cite{DMP_OrbitsO6arX}, we determine all the incidence matrices connected with $\mathscr{F}$ and the special orbit. For plane-line incidence matrices we obtain the numbers of distinct planes through distinct lines and, conversely, the numbers of distinct
lines lying in distinct planes. (By ``distinct planes'' we mean ``planes from distinct orbits'', and similarly for points and lines.) For point-line incidence matrices we obtain the numbers of distinct lines through distinct points and, vice versa, the numbers of distinct points lying on distinct lines. To obtain the needed numbers we prove relations connecting them based on the general properties of $\mathcal{O}_6$, see Section~\ref{sec:useful}. Then, using the properties of
$\mathscr{F}$, we consider intersections of lines generating orbits of $\mathscr{F}$ with distinct lines and planes.
This gives rise to some cubic and quartic equations, the numbers of solutions of which is equal to parameters of the incidence matrices, see Sections \ref{sec:incidL}, \ref{sec:lmuW&intersec}. Formulas for the numbers
of solutions are described.

The incidence submatrices considered in this paper are configurations in the sense of \cite{GroppConfig}, see Definition~\ref{def2_config} in Section \ref{subsec:J2}. They do not contain $2\times2$ submatrices whose entries are~1. The interest for such matrices is, for example, in the classical Zarankiewicz problem, see \cite{EK,RI,graphSurvey}, or in the construction of low-density parity-check (LDPC) codes, whose corresponding bipartite graph codes are free from the so-called 4-cycles \cite{BargZem,DGMP_BipGraph,HohJust}.

The paper is organized as follows. Section \ref{sec_prelimin} contains preliminaries. including the
description of $\mathscr{F}$, see Section \ref{subsec:knon orbits}. Some useful relations are given in Section \ref{sec:useful}. The incidence matrices and needed cubic equations for the special orbit  are obtained in Section~\ref{sec:incidL}. In Section \ref{sec:lmuW&intersec}, we consider intersections of lines generating orbits of $\mathscr{F}$ given parametrically with distinct lines and planes and form the
needed cubic and quartic equations. In Sections \ref{sec:even q}, \ref{sec:odd qne0}, and \ref{sec:q=0mod3}, the incidence matrices, respectively,  for even $q$, odd $q\not\equiv0\pmod3$, and $q\equiv0\pmod3$, are obtained.

\section{Preliminaries}\label{sec_prelimin}

\subsection{Twisted cubic}\label{subset_twis_cub}
We summarize the results on the twisted cubic of \cite{Hirs_PG3q} useful in this paper.

Let $\Pf(x_0,x_1,x_2,x_3)$ be a point of $\PG(3,q)$ with homogeneous coordinates $x_i\in\F_{q}$.
For $t\in\F_q^+$, let  $P(t)$ be a point such that
\begin{align}\label{eq2:P(t)}
  P(t)=\Pf(t^3,t^2,t,1)\text{ if }t\in\F_q;~~P(\infty)=\Pf(1,0,0,0).
\end{align}

Let $\C\subset\PG(3,q)$ be the \emph{twisted cubic} consisting of $q+1$ points $P_1,\ldots,P_{q+1}$  no four of which are coplanar.
We consider $\C$ in the canonical form
\begin{align}\label{eq2_cubic}
&\C=\{P_1,P_2,\ldots,P_{q+1}\}=\{P(t)\,|\,t\in\F_q^+\}.
\end{align}

A \emph{chord} of $\C$ is a line through a pair of real points of $\C$ or a pair of complex conjugate points. In the last  case it is an \emph{imaginary chord}. If the real points are distinct, it is a \emph{real chord}; if they coincide with each other, it is a \emph{tangent.}
Let $\TT_t$ be the tangent to $\C$ at $P(t)$; $\TT_t$  has a coordinate vector
\begin{align}\label{eq2:cvTang}
& L^{\T{tang}}_t=(t^4, 2t^3, 3t^2, t^2, -2t, 1),~t\in\F_q;~L^{\T{tang}}_\infty=(1, 0, 0, 0, 0, 0).
\end{align}

Let $\boldsymbol{\pi}(c_0,c_1,c_2,c_3)\subset\PG(3,q)$, be the plane with equation
\begin{align}\label{eq2_plane}
  c_0x_0+c_1x_1+c_2x_2+c_3x_3=0,~c_i\in\F_q.
\end{align}
The \emph{osculating plane} in the  point $P(t)\in\C$ is as follows:
\begin{align}\label{eq2_osc_plane}
&\pi_\T{osc}(t)=\boldsymbol{\pi}(1,-3t,3t^2,-t^3)\T{ if }t\in\F_q; ~\pi_\T{osc}(\infty)=\boldsymbol{\pi}(0,0,0,1).
\end{align}
 The $q+1$ osculating planes form the osculating developable $\Gamma$ to $\C$, that is a \emph{pencil of planes} for $q\equiv0\pmod3$ or a \emph{cubic developable} for $q\not\equiv0\pmod3$.

 An \emph{axis} of $\Gamma$ is a line of $\PG(3,q)$ which is the intersection of a pair of real planes or complex conjugate planes of $\Gamma$. In the last  case it is an \emph{imaginary axis}. If the real planes are distinct it is a \emph{real axis}; if they coincide with each other, it is a tangent to $\C$.

For $q\not\equiv0\pmod3$, the null polarity $\A$ \cite[Sections 2.1.5, 5.3]{Hirs_PGFF}, \cite[Theorem~21.1.2]{Hirs_PG3q} is given by
\begin{align}\label{eq2_null_pol}
&\Pf(x_0,x_1,x_2,x_3)\A=\boldsymbol{\pi}(x_3,-3x_2,3x_1,-x_0).
\end{align}

\begin{notation}\label{notation_1}
Throughout the paper, we consider $q\equiv\xi\pmod3$ with $\xi\in\{-1,0,1\}$. Many values depend on $\xi$ or have sense only for specific $\xi$.
If it is not clear by the context, we note this by remarks.
The following notation is used.
\begin{align*}
  &G_q && \T{the group of projectivities in } \PG(3,q) \T{ fixing }\C;\db  \\
&A^{tr}&&\T{the transposed matrix of }A;\db \\
&\#S&&\T{the cardinality of a set }S;\db\\
&\overline{AB}&&\T{the line through the points $A$ and }B;\db\\
&\triangleq&&\T{the sign ``equality by definition''}.
\end{align*}
\centerline{\textbf{Types $\pi$ of planes:}}
\begin{align*}
&\Gamma\T{-plane}  &&\T{an osculating plane of }\Gamma;\db \\
&d_\C\T{-plane}&&\T{a plane containing \emph{exactly} $d_\C$ distinct points of }\C,~d_\C\in\{0,2,3\};\db \\
&\overline{1_\C}\T{-plane}&&\T{a plane not in $\Gamma$ containing \emph{exactly} 1 point of }\C;\db \\
&\Pk&&\T{the list of possible types $\pi$ of planes},~\Pk\triangleq\{\Gamma,2_\C,3_\C,\overline{1_\C},0_\C\};\db\\
&\pi\T{-plane}&&\T{a plane of the type }\pi\in\Pk.
\end{align*}
\centerline{\textbf{Types $\pk$ of points with respect to the twisted cubic $\C$:}}
\begin{align*}
&\C\T{-point}&&\T{a point  of }\C;\db\\
&\mu_\Gamma\T{-point}&&\T{a point  off $\C$ lying on \emph{exactly} $\mu$ distinct osculating planes},\db\\
&&&\mu_\Gamma\in\{0_\Gamma,1_\Gamma,3_\Gamma\}\T{ for }\xi\ne0,~\mu_\Gamma\in\{(q+1)_\Gamma\}\T{ for }\xi=0;\db\\
&\Tr\T{-point}&&\T{a point  off $\C$  on a tangent to $\C$ for }\xi\ne0;\db\\
&\TO\T{-point}&&\T{a point  off $\C$ on a tangent and one osculating plane for }\xi=0;\db\\
&\RC\T{-point}&&\T{a point  off $\C$  on a real chord;}\db\\
&\IC\T{-point}&&\T{a point  on an imaginary chord (it is always off $\C$);}\\
&\Mk&&\T{the list of possible types $\pk$ of points},\db\\
&&&\Mk\triangleq\{\C,0_\Gamma,1_\Gamma,3_\Gamma,\Tr,\RC,\IC\}\T{ for }\xi\ne0,\db\\
&&&\Mk\triangleq\{\C,(q+1)_\Gamma,\TO,\RC,\IC\}\T{ for }\xi=0;\db\\
&\pk\T{-point}&&\T{a point of the type }\pk\in\Mk.
\end{align*}
\centerline{\textbf{Orbits under $G_q$:}}
\begin{align*}
&\N_\pi&&\T{the orbit of $\pi$-planes under }G_q,~\pi\in\Pk;\db\\
&\M_\pk&&\T{the orbit of $\pk$-points under }G_q,~\pk\in\Mk;\db\\
&\EnG\T{-line}&&\T{a line, external to the cubic $\C$, not in osculating planes,\db}\\
&&&\T{that is neither a chord nor an axis;}\db \\
&\O_6=\O_{\EnG}&&\T{the union (class) of all orbits of $\EnG$-lines}.
\end{align*}
\end{notation}

\begin{remark}
  The words ``nor an axis" are included to the definition of $\EnG$-line by the
context of \cite[Lemma 21.1.4]{Hirs_PG3q}.
\end{remark}

The following theorem summarizes the results from \cite{Hirs_PG3q} useful in this paper.
\begin{theorem}\label{th2_Hirs}
\emph{\cite[Chapter 21]{Hirs_PG3q}} The following properties of the twisted cubic $\C$ of \eqref{eq2_cubic} hold:
\begin{description}
  \item[(i)] The group $G_q$ acts triply transitively on $\C$;  $G_q\cong PGL(2,q)$ for $q\ge5$.

    The matrix $\MM$ corresponding to a projectivity of $G_q$ has the general form
  \begin{align}\label{eq2_M}
& \mathbf{M}=\left[
 \begin{array}{cccc}
 a^3&a^2c&ac^2&c^3\\
 3a^2b&a^2d+2abc&bc^2+2acd&3c^2d\\
 3ab^2&b^2c+2abd&ad^2+2bcd&3cd^2\\
 b^3&b^2d&bd^2&d^3
 \end{array}
  \right],~a,b,c,d\in\F_q,~ ad-bc\ne0.
\end{align}

  \item[(ii)] Under $G_q$, $q\ge5$, there are the following five orbits $\N_\pi$ of planes:
\begin{align*}
   &\N_1=\N_\Gamma=\{\Gamma\T{-planes}\},~\#\N_\Gamma=q+1;~\N_{2}=\N_{2_\C}=\{2_\C\T{-planes}\},~\#\N_{2_\C}=q^2+q;\db \\
 &\N_{3}=\N_{3_\C}=\{3_\C\T{-planes}\},~  \#\N_{3_\C}=(q^3-q)/6;~\N_{4}=\N_{\overline{1_\C}}=\{\overline{1_\C}\T{-planes}\},\db\\
 &\#\N_{\overline{1_\C}}=(q^3-q)/2;~\N_{5}=\N_{0_\C}=\{0_\C\T{-planes}\},~\#\N_{0_\C}=(q^3-q)/3.
 \end{align*}

  \item[(iii)] For $q\not\equiv0\pmod 3$, there are the following five orbits $\M_j$ of points:
  \begin{align}
&\M_1=\M_\C=\{\C\T{-points}\},~\M_2=\M_\Tr=\{\Tr\T{-points}\},~\M_3=\M_{3_\Gamma}=\{3_\Gamma\T{-points}\},\db\nt\\
&\M_4=\M_{1_\Gamma}=\{1_\Gamma\T{-points}\},~\M_5=\M_{0_\Gamma}=\{0_\Gamma\T{-points}\}.\db\nt\\
 &\T{For } q\equiv1\pmod 3,~ \M_{3_\Gamma}\cup\M_{0_\Gamma}=\{\RC\T{-points}\}, ~ \M_{1_\Gamma}=\{\IC\T{-points}\};\db\nt\\
 &\T{for } q\equiv-1\pmod 3,~\M_{3_\Gamma}\cup\M_{0_\Gamma}=\{\IC\T{-points}\},~
 \M_{1_\Gamma}=\{\RC\T{-points}\}. \db\nt\\
&  \M_j\A=\N_j,~\#\M_j=\#\N_j,~j=1,\ldots,5; \db\nt\\
&\M_\C\A=\N_\Gamma,~\M_\Tr\A=\N_{2_\C}, ~
\M_{3_\Gamma}\A=\N_{3_\C},~\M_{1_\Gamma}\A=\N_{\overline{1_\C}},~
\M_{0_\Gamma}\A=\N_{0_\C}.\label{eq2:pi(pk)}
\end{align}

\item[(iv)] For $q\equiv0\pmod 3$, the are the following five orbits $\M_j$ of points:
\begin{align*}
&\M_1=\M_\C=\{\C\T{-points}\},~\M_2=\M_{(q+1)_\Gamma}=\{(q+1)_\Gamma\T{-points}\},\db\\
&\#\M_\C=\#\M_{(q+1)_\Gamma}=q+1;~\M_3=\M_\TO=\{\TO\T{-points}\},~\#\M_\TO=q^2-1;\dbn\\
&\M_4=\M_\RC=\{\RC\T{-points}\},~M_5=\M_\IC=\{\IC\T{-points}\},\dbn\\
&\#\M_\RC=\#\M_\IC=(q^3-q)/2.
\end{align*}

  \item[(v)]  The lines of $\PG(3,q)$ can be partitioned into classes called $\O_i$ and $\O'_i=\O_i\A$, each of which is a union of orbits under $G_q$. The full list of the classes can be found in \emph{\cite[Lemma 21.1.4]{Hirs_PG3q}}. In particular, for all $q$, there is the class $\O_6=\O_{\EnG}=\{\EnG\T{-lines}\}$, $\#\O_6=\#\O_{\EnG}=(q^2-q)(q^2-1)$. If  $q\not\equiv0\pmod3,$ we have $\O_6=\O'_6=\O_6\A$.
\end{description}
\end{theorem}

\subsection{$\J_2$-free matrices and configurations}\label{subsec:J2}
We denote by $\J_2=\left[\begin{array}{cc}
                           1 & 1 \\
                           1 & 1
                         \end{array}
\right]$ the $2\times2$ matrix  consisting of all ones. A $01$-matrix not containing the submatrices $\J_2$ is called a $\J_2$\emph{-free matrix}.
Designs 2-$(v,k,1)$ and configurations are important examples of $\J_2$-free matrices.
\begin{definition}\label{def2_config}\cite{GroppConfig}
    A \emph{configuration} $(v_r,b_k)$  is an incidence structure of $v$ points and $b$ lines such that
 each line contains $k$ points, each point lies on $r$ lines, and
 two different points are connected by at most one line. If $v = b$ and, hence, $r = k$, the configuration is symmetric, denoted by $v_k$.
\end{definition}
\noindent For an introduction to configurations see \cite{DFGMP_SymConf,GroppConfig} and the references therein.

$\J_2$\emph-free matrices are known to be of interest, for instance,
in the classical Zarankiewicz problem, see \cite{EK,RI,graphSurvey}, but also in a recent
application to the low-density parity-check (LDPC)
codes, i.e. error correcting codes with a strongly sparse parity check matrix. The absence of the submatrices $\J_2$ allows to avoid the 4-cycles in the bipartite graphs corresponding to the codes, see e.g. \cite{BargZem,DGMP_BipGraph,HohJust} and the references therein.

\subsection{The known orbits of $\EnG$-lines from class $\O_6=\O_{\EnG}$}\label{subsec:knon orbits}
In this subsection we give results from \cite{DMP_OrbitsO6arX} useful in this paper.
\subsubsection{An orbit $\Os_\L$}\label{subsubsec:OL}
Let $Q_\beta$ and $Q_\infty$ be the points: $Q_\beta=\Pf(1,0,\beta,1), ~\beta\in\F_q$; $Q_\infty=\Pf(0,0,1,0)$. We consider the line $\ell_\L = \overline{Q_0Q_\infty}$ through the points $Q_\beta$ and $Q_\infty$. We have
\begin{equation}\label{eq3:ell0infdef}
 \ell_\L =\overline{\Pf(1,0,0,1)\Pf(0,0,1,0)}= \{\Pf(0,0,1,0), \Pf(1,0,\beta,1)\,|\,\beta\in\F_q\}.
\end{equation}
Let $\Os_\L$ be the orbit of $\ell_\L$ under $G_q$.
\begin{theorem}\label{th2:orbitell0inf}
\emph{\cite[Theorem 3.5]{DMP_OrbitsO6arX}}
\begin{description}
  \item[(i)] For $q\not\equiv0\pmod3$, we have $\Os_\L\subset\O_6=\O_{\EnG}$, i.e. the lines of $\Os_\L$ are $\EnG$-lines.
  \item[(ii)]
Let $q\equiv\xi\pmod3$. The orbit $\Os_L$ has size
\begin{align}\label{eq2:orbit0inf}
 \#\Os_\L=\left\{\begin{array}{@{\,}lccl}
                          (q^3-q)/3 & \T{if} & \xi=1, &q\T{ is even or }2\T{ is a non-cube in }\F_q;\\
                          (q^3-q)/12&\T{if} &\xi=1,  &q \T{ is odd and }2\T{ is a cube in }\F_q; \\
                          q^3-q& \T{if}&\xi=-1,  &q\T{ is even}; \\
                          (q^3-q)/2& \T{if}&\xi=-1,  &q\T{ is odd}.
                        \end{array}
 \right.
\end{align}
\end{description}
\end{theorem}

\subsubsection{Orbits $\Os_\mu$, $\mu\in \F_q^*\setminus\{1,1/9\}$}\label{subsubsece:lmu}
 Let $R_{\mu,\gamma}$  be the point: \\
  $ R_{\mu,\gamma}=\Pf(\gamma,\mu,\gamma,1)$, $\gamma\in\F_q^+$; $ R_{\mu,0}=\Pf(0,\mu,0,1)$, $R_{\mu,\infty}=\Pf(1,0,1,0)$,
 \begin{align}\label{eq2:mu}
 \mu\in \left\{\begin{array}{lcl}
                 \F_q^*\setminus\{1\} & \T{if}  &q\T{ is even or }q\equiv0\pmod3\\
                 \F_q^*\setminus\{1,1/9\}& \T{if}  &q\T{ is odd and }q \not\equiv0\pmod3
               \end{array}
 \right..
\end{align}
We consider the line $\ell_{\mu} = \overline{R_{\mu,0}R_{\mu,\infty}}$ through $R_{\mu,0}$ and $R_{\mu,\infty}$.
\begin{equation}\label{eq2:ellmu}
\ell_{\mu} = \overline{\Pf(0,\mu,0,1)\Pf(1,0,1,0)}= \{\Pf(\gamma,\mu,\gamma,1)|\gamma\in\F_q^+,~\mu\T{ is fixed}\}.
\end{equation}

Let $\Os_{\mu}$ be the orbit of $\ell_\mu$ under $G_q$.
\begin{theorem}\label{th2:orbitellmu}
\emph{\cite[Lemma 4.2,  Theorems 3.5, 5.2, 6.3]{DMP_OrbitsO6arX}}
\begin{description}
  \item[(i)] We have $R_{\mu,\gamma}\notin\pi_\T{osc}(\infty),~\gamma\in\F_q$.

  \item[(ii)] For all $q\ge5$, we have $\Os_\mu\subset\O_6=\O_{\EnG}$, i.e. the lines of $\Os_\mu$ are $\EnG$-lines.

  \item[(iii)]
Let $q\equiv\xi\pmod3$. Let the condition $\Upsilon_{q,\mu}$ be of the form
\begin{align}\label{eq2:Upsilon}
  \Upsilon_{q,\mu} \;:~\mu=-1/3,~q \equiv 1 \pmod {12}, ~1/3\T{ is a fourth power}.
\end{align}
The orbit $\Os_\mu$ has size
\begin{align}\label{eq2:orbitellmu}
 \#\Os_\mu=\left\{\begin{array}{@{\,}ll}
                          (q^3-q)/2& \T{if } q\T{ is even or }\mu \T{ is a non-square in }\F_q;\\
                          (q^3-q)/4& \T{if }\mu \T{ is a square in }\F_q\T{ and }\xi=0;\\
                          (q^3-q)/4&\T{if }q\T{ is odd},  \mu \T{ is a square in }\F_q,~\xi\ne0,\\
                          & \Upsilon_{q,\mu}\T{ does not hold};\\
                          (q^3-q)/12&\T{if }q\T{ is odd},~\xi\ne0,  \Upsilon_{q,\mu}\T{ holds}.
                        \end{array}
 \right.
\end{align}
\end{description}

\end{theorem}

\section{Useful relations}\label{sec:useful}

\begin{notation}\label{notation_2}
In addition to Notation \ref{notation_1}, for $\pi\in\Pk$, $\pk\in\Mk$, and an orbit $\Os$ of $\EnG$-lines, the following notation is used:
\begin{align*}
&\Pi_{\pi}&&\T{the number of $\pi$-planes through a line from the orbit } \Os;\db\\
&\Lambda_{\pi}&&\T{the number of lines from the orbit $\Os$ in a $\pi$-plane}; \db\\
&\Pb_{\pk}&&\T{the  number of $\pk$-points on a line from  the orbit }\Os;\db\\
&\Lb_{\pk}&&\T{the number of lines from  the orbit $\Os$ through a $\pk$-point}.\
 \end{align*}
\end{notation}

From now on, when we use the notations and the results of this section, as the line-orbit~$\Os$ we consider the orbit $\Os_\L$ or  $\Os_\mu$, see Section \ref{subsubsece:lmu}. The situation will be clear by the context.

Also, from now on, we consider $q\ge5$.

The following lemma is obvious, see also \cite[Lemma 4.1, Corollary 4.2]{DMP_PointLineInc}, \cite[Lemma~4.1]{DMP_PlLineIncJG}.
\begin{lemma}\label{lem3:orb the same}
\begin{description}
  \item[(i)]
  The number  of lines from an orbit $\Os$ in a plane of an orbit $\N_\pi$ is the same for all planes of~$\N_\pi$;
conversely, the number of planes from the orbit $\N_\pi$ through a line of the orbit $\Os$ is the same for all lines of $\Os$.

\item[(ii)]
  The number of lines from an orbit $\Os$ through a point of an orbit $\M_\pk$ is the same for all points of~$\M_\pk$.
And, vice versa, the number of points from the orbit $\M_\pk$ on a line of the orbit $\Os$ is the same for all lines of $\Os$.

\item[(iii)]
 \begin{equation}\label{eq3:obtainPi1}
  \mathrm{\Pi}_{\pi}=\frac{\mathrm{\Lambda}_{\pi}\cdot\# \N_\pi}{\#\Os},~\Pb_{\pk}=\frac{\Lb_{\pk}\cdot\#\M_\pk}{\#\Os};
  \end{equation}
  \begin{equation}
  \mathrm{\Lambda}_{\pi}=\frac{\mathrm{\Pi}_{\pi}\cdot\#\Os}{\#\N_\pi},
  ~\Lb_{\pk}=\frac{\Pb_{\pk}\cdot\#\Os}{\#\M_\pk}.\label{eq3:obtainLamb1}
  \end{equation}
\end{description}
\end{lemma}

Let $\ell$ be a line.
Let $\Pi_\pi(\ell)$ be the number of $\pi$-planes through $\ell$, $\pi\in\Pk$. Let $\Pb_\pk(\ell)$ is the number of $\pk$-points on $\ell$, $\pk\in\Mk$.

\begin{lemma} \label{lem3:tang&2cplane}
\begin{description}
  \item[(i)] For all $q\ge5$, the union of $q+1$ $\Gamma$-planes and $q^2+q$  $2_\C$-planes can be partitioned into $q+1$ pencils of planes such that each pencil consists of one $\Gamma$-plane and $q$ $2_\C$-planes and the axis of the pencil is a tangent to $\C$.

  \item[(ii)]  For all $q\ge5$, among $\EnG$-lines, only lines lying in $2_\C$-planes can intersect tangents.
  If $q\not\equiv0\pmod3$, $\EnG$-lines lying in $2_\C$-planes intersect tangents in $\Tr$-points.
  If $q\equiv0\pmod3$, $\EnG$-lines lying in $2_\C$-planes intersect tangents in $\TO$-points.

  \item[(iii)] Let $\ell$ be an $\EnG$-line. Then  $
  \Pb_{\Tr}(\ell)=\Pi_{2_\C}(\ell)$ if $q\not\equiv0\pmod3$ and $
   \Pb_{\TO}(\ell)=\Pi_{2_\C}(\ell)$ if $q\equiv0\pmod3.$
\end{description}
\end{lemma}

\begin{proof}
\begin{description}
  \item[(i)]
By \cite[Table 1]{DMP_PlLineIncJG}, there is one tangent in every $2_\C$-plane and, conversely, there are $q$  $2_\C$-planes through each tangent;
  in addition, there is one tangent in every $\Gamma$-plane and, vice versa, there is one $\Gamma$-plane through each tangent.

  \item[(ii)]
By definition, an $\EnG$-line does not belong to a $\Gamma$-plane and does not contain a $\C$-point. Also, we  take into account  the case~(i) and the fact that,
   for $q\equiv0\pmod3$, $\Gamma$-planes  form a pencil whose axis can intersect only lines lying in a $\Gamma$-plane.

  \item[(iii)] We take into account the cases (i) and (ii) and their proofs.  Also, by \cite[Proposition 5.6]{DMP_PlLineIncJG}, for a real chord $\mathcal{RC}$ and the two $\Gamma$-planes in its common points with  $\C$, the following holds:
every $2_\C $-plane through $\mathcal{RC}$ intersects one of these $\Gamma$-planes
in its tangent and the other in a non-tangent unisecant. The above means that the
 $2_\C $-planes through an $\ell_\mu$-line and the tangents intersecting it complete the set of  $2_\C $-planes containing this line. \qedhere
 \end{description}
\end{proof}

\begin{proposition}\label{prop3:PiG=LambG=PbC=LbC=0}
For any orbit $\Os$ of $\EnG$-lines the following holds.
 \begin{align}
&\Pi_\Gamma=\Lambda_{\Gamma}=\Pb_{\C}=\Lb_{\C}=0,\T{ for all }q.\db\label{eq3:=0}\\
&\Pb_{\Tr}=\Pi_{2_\C}\T{ if }q\not\equiv0\pmod3;~
   \Pb_{\TO}=\Pi_{2_\C}\T{ if }q\equiv0\pmod3.\label{eq3:2C=t}
\end{align}
\end{proposition}

\begin{proof}
The assertions follows from the definition of $\EnG$-lines and Lemmas \ref{lem3:orb the same}, \ref{lem3:tang&2cplane}.
\end{proof}

For $q\not\equiv0\pmod3$, we denote $\widetilde{\ell}=\ell\A$ the image of a line $\ell=\overline{P_1P_2}$ such that $\ell\A=P_1\A\cap P_2\A$ where $P_1$, $P_2$ are two distinct points of $\ell$ and $P_1\A$, $P_2\A$ are planes. Obviously, $\ell=\widetilde{\ell}\A$. By Theorem \ref{th2_Hirs}(iii), including \eqref{eq2:pi(pk)}, we have the following lemma.

\begin{lemma}\label{lem3:UPiPk}
  Let $q\not\equiv0\pmod3$. The following holds:
  \begin{align*}
&\Pi_{2_\C}(\widetilde{\ell})=\Pb_{\Tr}(\ell),~ \Pi_{3_\C}(\widetilde{\ell})=\Pb_{3_\Gamma}(\ell),~\Pi_{\overline{1_\C}}(\widetilde{\ell})=\Pb_{1_\Gamma}(\ell),~\Pi_{0_\C}(\widetilde{\ell})=\Pb_{0_\Gamma}(\ell).
  \end{align*}
\end{lemma}

\begin{lemma}\label{lem3:EnG=ext}
  Let $\ell$ be an $\EnG$-line. Let $\widetilde{\ell}=\ell\A$. Then $\ell=\widetilde{\ell}\A$ and
\begin{align}\label{eq3:EnG=ext}
&\Pi_{\overline{1_\C}}(\ell)+2\Pi_{2_\C}(\ell)+3\Pi_{3_\C}(\ell)=q+1,\T{ for all }q;\db\\
&\Pi_{\overline{1_\C}}(\widetilde{\ell})+2\Pi_{2_\C}(\widetilde{\ell})+3\Pi_{3_\C}(\widetilde{\ell})=q+1,~q\not\equiv0\pmod3;\label{eq3:EnG=ext2}\db\\
  &\Pb_{1_\Gamma}(\ell)+2\Pb_{\Tr}(\ell)+3\Pb_{3_\Gamma}(\ell)=q+1,~q\not\equiv0\pmod3;\db \label{eq3:EnG=ext3}\\
  &\Pb_{1_\Gamma}(\widetilde{\ell})+2\Pb_{\Tr}(\widetilde{\ell})+3\Pb_{3_\Gamma}(\widetilde{\ell})=q+1,~q\not\equiv0\pmod3. \label{eq3:EnG=ext4}
  \end{align}
\end{lemma}

\begin{proof}
  As $\ell$ is external with respect to $\C$ we may use \cite[Theorem 4.3]{DMP_PlLineIncJG}. Then we apply \eqref{eq3:=0} and obtain \eqref{eq3:EnG=ext}. As $\widetilde{\ell}$ also is an $\EnG$-line,  \eqref{eq3:EnG=ext2} can be obtained similarly.

  Then  on \eqref{eq3:EnG=ext2}   we act by the null-polarity $\A$, taking into account that $\ell\A\A=\ell$ and applying Lemma \ref{lem3:UPiPk}, that implies \eqref{eq3:EnG=ext3}. Finally, we act by  $\A$ on \eqref{eq3:EnG=ext} and obtain \eqref{eq3:EnG=ext4}.
\end{proof}

\begin{proposition}
  For an orbit $\Os$  the following holds.
 \begin{align}\label{eq3:Pi1Pi2Pi3}
&\Pi_{\overline{1_\C}}+2\Pi_{2_\C}+3\Pi_{3_\C}=q+1,\T{ for all }q\db\\
&\Pi_{0_\C}=\Pi_{2_\C}+2\Pi_{3_\C},\T{ for all }q;\label{eq3:Pi1Pi2Pi3b}\db\\
&\Pb_{1_\Gamma}+2\Pb_{\Tr}+3\Pb_{3_\Gamma}=q+1,~q\not\equiv0\pmod3,\label{eq3:Pi1Pi2Pi3c}\db\\
&\Pb_{0_\Gamma}=\Pb_{\Tr}+2\Pb_{3_\Gamma},~q\not\equiv0\pmod3\label{eq3:Pi1Pi2Pi3d}.
\end{align}
\end{proposition}

\begin{proof}
We use \eqref{eq3:=0} and Lemma \ref{lem3:orb the same}(i)(ii), that gives \eqref{eq3:Pi1Pi2Pi3} and  \eqref{eq3:Pi1Pi2Pi3c}. Then we again apply \eqref{eq3:=0}  and the facts that there are $q+1$ planes through a line and $q+1$ points on a line. This gives \eqref{eq3:Pi1Pi2Pi3b} and  \eqref{eq3:Pi1Pi2Pi3d}.
\end{proof}

\section{The point-line and plane-line incidence submatrices for the orbit $\Os_\L$, $q\not\equiv0\pmod3$}\label{sec:incidL}
In this section $q\not\equiv0\pmod3$.
We consider the line $\ell_\L$ and its orbit $\Os_\L$, see Section \ref{subsubsec:OL}. When we use the notations and the results of Section \ref{sec:useful}, we take $\Os_\L$ as the line orbit $\Os$.

\begin{lemma}\label{lem4:Tpoint=2Gammapoint}
  For $q\not\equiv0\pmod3$,  every $\Tr$-point lies in exactly two osculating planes.
\end{lemma}

\begin{proof}
  The assertion follows from \cite[Table 1]{BDMP-TwCub}.
\end{proof}

\begin{lemma}\label{lem4:cubic_equation}
 Let the point $Q_\beta=\Pf(1,0,\beta,1)$ lie in an osculating plane $\pi_\T{osc}(t)$, $t\in\F_q$. Then the values of $\beta$ and $t$ satisfy the following cubic equation
 \begin{equation}\label{eq4:cubic_equation}
F_\beta(t)=  t^3-3\beta t^2- 1=0,~\beta,t\in\F_q.
 \end{equation}
\end{lemma}

 \begin{proof}
  By \eqref{eq2_osc_plane}, $\pi_\T{osc}(t)=\boldsymbol{\pi}(1,-3t,3t^2,-t^3)$, $t\in\F_q$.
   \end{proof}

We denote by $\V_m$ the number of $\beta\in\F_q$ such that the cubic equation $F_\beta(t)$ \eqref{eq4:cubic_equation}
has exactly $m$ distinct solutions $t$ in $\F_q$, $m=0,1,2,3$.

\begin{lemma}\label{lem4:Nm&Pb}
   For the orbit $\Os_\L$, the following holds.
  \begin{equation}\label{eq4:solut&incid}
    \Pb_\Tr=\V_2+1,~\Pb_{1_\Gamma}=\V_1.
      \end{equation}
\end{lemma}

\begin{proof}
By Lemma \ref{lem4:cubic_equation}, if, for a fixed $\beta$,  the equation $F_\beta(t)$ \eqref{eq4:cubic_equation} has exactly $m$ distinct solutions $t$ in $\F_q$ then the point $Q_\beta$ belongs to exactly $m$ distinct osculating planes. So, the line $\ell_\L\setminus\{Q_\infty\}$ contains $\V_m$ points belonging to exactly $m$ distinct osculating planes. In particular, if $m=2$, they are $\Tr$-points, see Lemma \ref{lem4:Tpoint=2Gammapoint}.

Also, $Q_\infty$ lies on the tangent to $\C$ at the point $\Pf(0,0,0,1)$ with equations
$x_0=x_1=0$, see   \cite[Lemma 5.2]{DMP_OrbLineMedit}. So, $Q_\infty$ is a $\Tr$-point.
\end{proof}

Remind that over $\F_q$, 
the equation $x^3=c$ has a unique solution if $q\equiv-1\pmod3$ or three distinct solutions if $q\equiv1\pmod3$ \cite[Section 1.5]{Hirs_PGFF}.

For $a\in\F_q$, the \emph{quadratic character} $\eta(a)$ is equal to 1 (resp. -1) if $a$ is a square (resp. non-square) in $\F_q^*$. Also, $\eta(0)=0$.
We denote
\begin{align*}
   & \Nk_q\triangleq\#\{\beta\,|\,\eta(1+4\beta^3)=-1,~\beta\in\F_q\}.
\end{align*}

  \begin{lemma}\label{lem4:Nmq_odd}
    Let $q$ be odd. Let $q\equiv\xi\pmod3$. For the orbit $\Os_\L$, we have
   \begin{align*}
  &\textbf{\emph{(i)}} ~~\,\V_2=1,~ \V_1=(q-1)/2,\T{ if }\xi=-1;\db\\
  &\textbf{\emph{(ii)}} ~\, \V_2=0,~\V_1=\Nk_q,\T{ if }\xi=1,~2\T{ is a non-cube in }\F_q;
  \db\\
  &\textbf{\emph{(iii)}}\,\V_2=3,~\V_1=\Nk_q,\T{ if }\xi=1,~2\T{ is a cube in }\F_q.
   \end{align*}
 \end{lemma}

 \begin{proof}
Let the discriminant $\Delta$ and the Hessian $H(T)$ of the equation $F_\beta(t)$  \eqref{eq4:cubic_equation} be as in \cite[Section 1.8, Lemma 1.18, Theorem 1.28]{Hirs_PGFF}.
We have
   $  \Delta=-27(1+4\beta^3),~H(T)=\beta^2T^2+T-\beta=0,~\beta\in\F_q$. The roots of $H(T)$ are $T=(-1\pm\sqrt{1+4\beta^3})/2\beta^2$.

   For the calculation of $\V_m$ we use \cite[Theorem 1.34, Table~1.3]{Hirs_PGFF}.

\textbf{(i)} Let $\Delta=0$. Then  $\beta=\sqrt[3]{-1/4}$ and  $F_\beta(t)$ has exactly two distinct roots: $-\beta$ and $-2\beta$. This implies $\V_2=1$.

Let $\Delta\ne0$. Then $\beta\ne\sqrt[3]{-1/4}$.
 If $\beta$ runs over $\F_q^*\setminus\sqrt[3]{-1/4}$ the value $1+4\beta^3$ runs over $\F_q^*\setminus\{1\}$ where there are exactly $(q-3)/2$ squares for which $H(T)$ has two solutions in $\F_q$ and $F_\beta(t)$ has one solution in $\F_q$. This gives $\V_1=(q-3)/2+1$ where ``+1'' takes into account $\beta=0$ when $F_\beta(t)$ has the unique root $t=1$.

\textbf{(ii)} As $-1/4$ is a non-cube in $\F_q$, always  $\Delta\ne0$ that implies $\V_2=0$.
If $\beta$ runs over $\F_q^*$ the value $1+4\beta^3$ takes $\Nk_q$ non-squares for which $H(T)$ has no solutions in $\F_q$ and $F_\beta(t)$ has exactly one solution in $\F_q$. Thus, $\V_1=\Nk_q$ as $F_\beta(t)$ has 3 distinct roots if $\beta=0$.

\textbf{(iii)}
Let $\Delta=0$. Then  $\beta=\sqrt[3]{-1/4}=\{\beta_1,\beta_2,\beta_3\}$. For every $i$, $F_{\beta_i}(t)$ has two distinct roots  that gives $\V_2=3$.
Let $\Delta\ne0$. Similarly to the case (ii) we obtain $\V_1=\Nk_q$.
\end{proof}

For even $q$, let $\mathrm{Tr}_2(a)$ be the absolute trace of $a\in\F_q$.  We denote
\begin{align*}
   &\Tk_q\triangleq\#\{\beta\,|\,\mathrm{Tr}_2(\beta^3)=1,\beta\in\F_q,~q=2^{2m}\}.
\end{align*}
By the context of \cite[Section 4]{CePa},
\begin{align}\label{eq4:tau}
\Tk_q=2^{2m-1}+(-2)^{m}=\frac{1}{2}q+(-1)^m\sqrt{q},~q=2^{2m}.
\end{align}

 \begin{lemma}\label{lem4:Nmq_even}
    Let $q$ be even. Let $q\equiv\xi\pmod3$. For the orbit $\Os_\L$, we have
   \begin{align*}
  &\textbf{\emph{(i)}} ~~\V_2=0,~\V_1=q/2,\T{ if }q=2^{2m+1},~\xi=-1;\db\\
  &\textbf{\emph{(ii)}} ~\V_2=0,~\V_1=\Tk_q,\T{ if }q=2^{2m},~\xi=1.
   \end{align*}
 \end{lemma}

\begin{proof}
For even $q$, $F_\beta(t)=  t^3+\beta t^2+1=0$, $H(T)=\beta^2T^2+T+\beta=0$, $\Delta\ne0$. The replacement $T$ by $T/\beta^2$ implies $\overline{H}(T)\triangleq T^2+T+\beta^3=0$.
By \cite[Sections 1.2(iv), 1.4]{Hirs_PGFF}, $\overline{H}(T)$ has 2 (resp. 0) roots in $\F_q$ if $\mathrm{Tr}_2(\beta^3)=0$ (resp. $\mathrm{Tr}_2(\beta^3)=1$).

For the calculation of $\V_m$ we use \cite[Theorem 1.34]{Hirs_PGFF}. As $\Delta\ne0$, we have $\V_2=0.$

\textbf{(i)}
If $\beta$ runs over $\F_q^*$ then $\mathrm{Tr}_2(\beta^3)$ is equal to 0 $(q-2)/2$ times, $\overline{H}(T)$ (and hence $H(T)$) has two roots in $\F_q$, and $F_\beta(t)$ has one root in $\F_q$. So, $\V_1=(q-2)/2+1$ where ``+1'' takes into account that $F_0(t)$ has one root in $\F_q$.

\textbf{(ii)}
If $\beta$ runs over $\F_q^*$ then $\mathrm{Tr}_2(\beta^3)$ is equal to 1 $\Tk_q$ times, $H(T)$  has no roots in $\F_q$, and $F_\beta(t)$ has one root in $\F_q$. So, $\V_1=\Tk_q$ as $F_0(t)$ has 3 distinct roots in $\F_q$.
\end{proof}

\begin{lemma}\label{lem4:L=LU}
Let $q\not\equiv0\pmod3$. We have $\L=\L\A$.
\end{lemma}

\begin{proof}
  We have $\L\A=Q_0\A\cap Q_\infty\A=\boldsymbol{\pi}(1,0,0,1)\cap\boldsymbol{\pi}(0,-3,0,0)$. One sees that $Q_0,Q_\infty\in\boldsymbol{\pi}(1,0,0,1)$ and $Q_0,Q_\infty\in\boldsymbol{\pi}(0,-3,0,0)$; this implies $Q_0,Q_\infty\in\L\A.$
\end{proof}
For the point type $\pk\in\{\Tr,3_\Gamma,1_\Gamma,0_\Gamma\}$,
let $\pi(\pk)$ be the plane type such that $\M_\pk\A=\N_{\pi(\pk)}$ in accordance to \eqref{eq2:pi(pk)}. We have
\begin{align}\label{eq4:pipk}
 \pi(\Tr)=2_\C,~\pi(3_\Gamma)=3_\C,~\pi(1_\Gamma)=\overline{1_\C},~\pi(0_\Gamma)=0_\C.
\end{align}
\begin{theorem}
For the submatrices of the point-line and plane-line incidence matrices with respect to the orbit $\Os_\L$, the values $\Pb_{\pk}$,
$\Pi_{\pi}$, $\Lb_{\pk}$, and $\Lambda_{\pi}$ (see Notation \emph{\ref{notation_2}}), are as in Table~\emph{\ref{tabO_L}} and \eqref{eq3:=0}. In that, $\Pb_{\pk}=\Pi_{\pi(\pk)}$  and $\Lb_{\pk}=\Lambda_{\pi(\pk)}$ with $\pi(\pk)$ in accordance to \eqref{eq4:pipk}. The sizes of orbits $\#\M_\pk=\#\N_{\pi(\pk)}$ are noted in the top of columns.
\end{theorem}

\begin{proof}
  We obtain the values of $\Pb_\Tr$ and $\Pb_{1_\Gamma}$ using \eqref{eq4:tau} and Lemmas \ref{lem4:Nm&Pb}--\ref{lem4:Nmq_even}. Then we calculate $\Pb_{3_\Gamma}$ by \eqref{eq3:Pi1Pi2Pi3c}, $\Pb_{0_\Gamma}$ by \eqref{eq3:Pi1Pi2Pi3d}, and all $\Lb_\pk$ by \eqref{eq3:obtainLamb1}. To prove  $\Pb_{\pk}=\Pi_{\pi(\pk)}$  and $\Lb_{\pk}=\Lambda_{\pi(\pk)}$ , we use Theorem \ref{th2_Hirs}(iii) and Lemmas~\ref{lem3:UPiPk}, \ref{lem4:L=LU}.
\end{proof}

\begin{table}[h]
\caption{Values $\Pb_{\pk}=\Pi_{\pi(\pk)}$ (the  number of $\pk$-points on a line from  $\Os_\L$=the  number of $\pi(\pk)$-planes through a line from $\Os_\L$, top entry)
and $\Lb_{\pk}=\Lambda_{\pi(\pk)}$ (the number of lines from  $\Os_\L$ through a $\pk$-point=the number of lines from  $\Os_\L$ in a $\pi(\pk)$-plane, bottom entry) for the point-line and plane-line incidence submatrices regarding the orbit $\Os_{\L}$; $q\equiv\xi\pmod3$,\newline
$\blacktriangledown$ means ``2 is a non-cube'', $\blacktriangle$ notes ``2 is a cube'', $\pi(\pk)$ see in \eqref{eq4:pipk}}
\centering
\begin{tabular}{ccccc}\hline
$\#\M_\pk\rightarrow$&$q^2+q$&$\frac{1}{6}(q^3-q)$&$ \frac{1}{2}(q^3-q)$&$\frac{1}{3}(q^3-q)$\\
$q$&$\Pb_\Tr=\Pi_{2_\C}$&$\Pb_{3_\Gamma}=\Pi_{3_\C}$&$\Pb_{1_\Gamma}=\Pi_{\overline{1_\C}}$&$\Pb_{0_\Gamma}=\Pi_{0_\C}$\\
$\xi$&$\Lb_\Tr=\Lambda_{2_\C}$&$\Lb_{3_\Gamma}=\Lambda_{3_\C}$&$\Lb_{1_\Gamma}=\Lambda_{\overline{1_\C}}$&$\Lb_{0_\Gamma}=\Lambda_{0_\C}$\\
$\#\Os_{\L}$\\\hline

odd $q$ &$2$&$ \frac{1}{6}(q-5)$&$ \frac{1}{2}(q-1)$&$ \frac{1}{3}(q+1)$\\
$\xi=-1$&$q-1$&$\frac{1}{2}(q-5)$&$\frac{1}{2}(q-1)$&$\frac{1}{2}(q+1)$\\
$\frac{1}{2}(q^3-q)$&\\

$q=2^{2m+1}$ &$1$&$ \frac{1}{6}(q-2)$&$ \frac{1}{2}q$&$ \frac{1}{3}(q+1)$\\
$\xi=-1$&$q-1$&$q-2$&$ q$&$ q+1$\\
$q^3-q$&\\

odd $q$,\,$\blacktriangledown$ &$1$&$ \frac{1}{3}(q-1-\Nk_q)$&$ \Nk_q$&$\frac{1}{3}(2q+1-2\Nk_q)$\\
$\xi=1$&$\frac{1}{3}(q-1)$&$\frac{2}{3}(q-1-\Nk_q)$&$\frac{2}{3}\Nk_q$&$\frac{1}{3}(2q+1-2\Nk_q)$\\
$\frac{1}{3}(q^3-q)$&\\

odd $q$,\,$\blacktriangle$ &$4$&$ \frac{1}{3}(q-7-\Nk_q)$&$ \Nk_q$&$\frac{2}{3}(q-1-\Nk_q)$\\
$\xi=1$&$\frac{1}{3}(q-1)$&$\frac{1}{6}(q-7-\Nk_q)$&$\frac{1}{6}\Nk_q$&$\frac{1}{6}(q-1-\Nk_q)$\\
$\frac{1}{12}(q^3-q)$&\\

$q=2^{2m}$ &$1$&$ \frac{1}{3}(\frac{1}{2}q-1-(-2)^m)$&$ \frac{1}{2}q+(-2)^m$&$\frac{1}{3}(q+1+(-2)^{m+1})$\\
$\xi=1$&$\frac{1}{3}(q-1)$&$\frac{2}{3}(\frac{1}{2}q-1-(-2)^m)$&$\frac{2}{3}(\frac{1}{2}q+(-2)^m)$&$\frac{1}{3}(q+1+(-2)^{m+1})$\\
$\frac{1}{3}(q^3-q)$&\\\hline
\end{tabular}
\label{tabO_L}
\end{table}

\begin{example}\label{ex5:odd}
Let $q\equiv1\pmod3$. Let $q$ be odd.  By computer search, using the system MAGMA \cite{Magma}, we obtained Table \ref{tab5.1} where one can do the following observations. We denote $\delta_q=2\Nk_q-q$. In Table \ref{tab5.1}, $\delta_q$ is odd, $\delta_q\equiv-1\pmod3$. Let  $q'\equiv q''\pmod9$; for both $q'$ and $q''$ 2 is a cube or a non-cube. Then  $\delta_q'\equiv\delta_q''\pmod9$.

In the bottom part of Table \ref{tab5.1} (where $2$ is a cube in $\F_q$) all the values of $\Nk_q$ are even providing integer values in  Table~\ref{tabO_L}.

\begin{table}[h]
\centering
\caption{The values $\Nk_q$, odd $q\equiv1\pmod3$, $7\le q\le907$}
$
\begin{array}
{c|cccccccccccc}\hline
\multicolumn{13}{c}{\T{$2$ is a non-cube in $\F_q$}}\\\hline
q&7&61&169&331&547&19&73&181&829&49&103&211\\
2\Nk_q-q&-1&-1&-1&-1&-1&-7&-7&-7&-7&-13&-13&-13\\\hline

q&373&859&97&151&421&907&163&487&409&571&787&349\\
2\Nk_q-q&-13&-13&-19&-19&-19&-19&-25&-25&-31&-31&-31&-37\\\hline

q&673&523&631&607&661&769&13&67&337&823&37&199\\
2\Nk_q-q&-37&-43&-43&-49&-49&-49&5&5&5&5&11&11\\\hline

q&361&577&79&241&619&139&193&463&271&379&541&757\\
2\Nk_q-q&11&11&17&17&17&23&23&23 &29&29&29&29\\\hline

q&313&367&853&751 &613&883&709&877\\
2\Nk_q-q&35&35&35&41&47&47&53&59\\\hline

\multicolumn{13}{c}{\T{$2$ is a cube in $\F_q$}}\\\hline
q&109&433&31&25&457&307&739&121&229&223&439&289\\
2\Nk_q-q&-1&-1&-7&-13&-13&-19&-19&-25&-25&-31&-31&-37\\\hline

q&397&643&529&841&43&691&157&127& 343&277&601&283\\
2\Nk_q-q&-37&-43&-49&-61&5&5&11&17&17&23&23&29\\\hline

q&499&727&625&733&811\\
2\Nk_q-q&29&41&47&47&53\\\hline
\end{array}
$
\label{tab5.1}
\end{table}
\end{example}
\section{Connections of lines $\ell_\mu$  with $\Gamma$-, $2_\C$-, and  $\overline{1_\C}$-planes and tangents to cubic $\C$ }
\label{sec:lmuW&intersec}
 Let the line $\ell_\mu$ be as in Section \ref{subsubsece:lmu}, in particular, $\mu\in\F_q^*\setminus\{1,1/9\}$ that for even $q$ and $q\equiv0\pmod3$ naturally reduces to $\mu\in\F_q^*\setminus\{1\}$.

\subsection{Intersections of lines $\ell_\mu$  with osculating planes, $q\not\equiv0\pmod3$}
 \begin{lemma}\label{lem5:R&RWmuinf}
Let $q$ be odd, $q\not\equiv0\pmod3$. Then the point $R_{\mu,\infty}=\Pf(1,0,1,0)$ is a $1_\Gamma$-point if  $q\equiv-1\pmod3$ and a $3_\Gamma$-point if $q\equiv1\pmod3$.
\end{lemma}

\begin{proof}
   We have $R_{\mu,\infty}\in\pi_\T{osc}(\infty)$ for all $q$, see \eqref{eq2_osc_plane}.
If  $q\equiv-1\pmod3$ then $R_{\mu,\infty}\notin\pi_\T{osc}(t)$ with $t\in\F_q$. If
 $q\equiv1\pmod3$ then $R_{\mu,\infty}\in\pi_\T{osc}(\pm\sqrt{-1/3})$.
\end{proof}

\begin{lemma}\label{lem5:cubic_equation}
Let $q\not\equiv0\pmod3$. Let $\gamma,t\in\F_q$. Let the point $R_{\mu,\gamma}=\Pf(\gamma,\mu,\gamma,1)$ belong to the osculating plane $\pi_\T{osc}(t)$. Then the values of $\mu,\gamma,$ and $t$ satisfy the cubic equation
\begin{align}\label{eq5:cubEq}
&  \Phi_{\mu,\gamma}(t)=t^3-3\gamma t^2+3\mu t-\gamma=0,~\gamma\in\F_q,~ \mu\in\F_q^*\setminus\{1,\frac{1}{9}\}, ~q\not\equiv0\pmod3.
\end{align}
\end{lemma}

\begin{proof}
 We have $\pi_\T{osc}(t)=\boldsymbol{\pi}(1,-3t,3t^2,-t^3)$, $t\in\F_q$, that implies the assertions.
   \end{proof}

   We denote $\Nb_1(\mu)$ the number of $\gamma\in\F_q,$ such that the equation $\Phi_{\mu,\gamma}(t)$  \eqref{eq5:cubEq} has exactly one solution $t$ in $\F_q$.

\subsection{Intersections  of lines $\ell_\mu$  with tangents to  cubic $\C$; lines $\ell_\mu$ and $2_\C$-planes}
The coordinate vector $L_\mu$ of $\ell_{\mu}$  is
$ L_\mu = (\mu,0,1,-\mu,0,1)$.
By \cite[Section 15.2]{Hirs_PG3q}, \eqref{eq2:cvTang},  the mutual invariant of $\ell_\mu$ and a tangent $\TT_t$ is:
\begin{align}\label{eq5:mutInvarW}
  &\varpi(\ell_\mu,\TT_t)=t^4-(3\mu-1) t^2+\mu,~t\in\F_q;~\varpi(\ell_\mu,\TT_\infty)=1\ne0.
\end{align}
Two lines   intersect if and only if their mutual invariant  is equal to zero. Thus, $\ell_\mu$  does not intersect $\TT_\infty$ for any $\mu$; we may consider only intersections  with $\TT_t$ for $t\in\F_q$.

Let $\mathfrak{n}_{q}(\mu)$ be
 the number of solutions in  $\F_q$ of the equation
 \begin{align}\label{eq5:eqtang}
 \varpi(\ell_\mu,\TT_t)=  t^4-(3\mu-1) t^2+\mu=0,~t\in\F_q.
 \end{align}
Useful results on $\mathfrak{n}_{q}(\mu)$ can be found in \cite[Section 1.11]{Hirs_PGFF}.
 By above, in particular by Lemma \ref{lem3:tang&2cplane}, Propostion \ref{prop3:PiG=LambG=PbC=LbC=0}, and \eqref{eq3:2C=t}, the following lemma holds.
 \begin{lemma}\label{lem5:Tpoints}
 Both the number of $\Tr$-points on an $\ell_\mu$-line and the number of $2_\C$-planes containing the line are equal to $\mathfrak{n}_{q}(\mu)$.
 \end{lemma}

\subsection{Line $\ell_\mu$ and $\overline{1_\C}$-planes}
 We consider the cubic equation regarding $t$.
\begin{align}\label{eq5:cub eq}
 \widetilde{\Phi}_{\mu,c}(t)\triangleq t^3+ct^2-t-\mu c=0,~t\in\F_q,~c\in\F_q^*,~\mu\T{ is as in }\eqref{eq2:mu}.
\end{align}
 We denote $\widetilde{\Nb}_1(\mu)$  the number of $c\in\F_q^*$ such that the equation $\widetilde{\Phi}_{\mu,c}(t)$ \eqref{eq5:cub eq}  has exactly one solution $t$ in $\F_q$.

\begin{lemma}\label{lem5:1Cplanes}
The number of\/ $\overline{1_\C}$-planes containing a line $\ell_\mu$ is equal to $\widetilde{\Nb}_1(\mu)$ if $\mu$ is a square in $\F_q$ and
$\widetilde{\Nb}_1(\mu)+1 $ otherwise.
\end{lemma}

\begin{proof}
 Let $\widetilde{\pi}= \boldsymbol{\pi}(c_0,c_1,c_2,c_3)$ be a plane containing $\ell_\mu$.
As $R_{\mu,0},R_{\mu,\infty}\in\widetilde{\pi}$, we have  $\widetilde{\pi}= \boldsymbol{\pi}(c_0,c_1,-c_0,-\mu c_1)$.

Let $c_0=0$. Then $c_1\ne0$, $\widetilde{\pi}= \boldsymbol{\pi}(0,1,0,-\mu)$, $P(\infty)\in\widetilde{\pi}$. For $t\in\F_q$,  $P(t)\in\widetilde{\pi}$ if  $t^2=\mu$. So,
$\widetilde{\pi}$ is a $\overline{1_\C}$ (resp. $3_\C$)-plane if $\mu$ is a non-square (resp. square) in $\F_q$.

Let $c_1=0$. Then $c_0\ne0$, $\widetilde{\pi}= \boldsymbol{\pi}(1,0,-1,0)$, $P(\infty)\notin\widetilde{\pi}$. For $t\in\F_q$,  $P(t)\in\widetilde{\pi}$ if $t^3=t$. So,
$\widetilde{\pi}$ is a $3_\C$-plane.

Let $c_0,c_1\ne0$. Then $\widetilde{\pi}= \boldsymbol{\pi}(1,c,-1,-\mu c)$,
$c=c_1/c_0\in\F_q^*$, $P(\infty)\notin\widetilde{\pi}$. The point $P(t)$ of $\C$, $t\in\F_q$, lies in $\widetilde{\pi}$ if and only if $t$ satisfies $\widetilde{\Phi}_{\mu,c}(t)$ \eqref{eq5:cub eq}. If $\widetilde{\Phi}_{\mu,c}(t)$
has exactly one solution $t$ in $\F_q$ then $\widetilde{\pi}$ contains exactly one point $P(t)$, i.e. $\widetilde{\pi}$ is a $\overline{1_\C}$-plane.
\end{proof}

 \section{Point-line and plane-line incidence submatrices for orbits $\Os_\mu$, even  $q$}\label{sec:even q}
In this section we assume that $q$ is even. This implies the natural simplification of some elements of Sections \ref{subsubsece:lmu} and \ref{sec:lmuW&intersec}. For even $q$ we have
$\mu\in\F_q^*\setminus\{1\}$ and
\begin{align}
 & \varpi(\ell_\mu,\TT_t)=t^4+(\mu+1) t^2+\mu;\db\label{eq6:tangintersec}\\
&  \Phi_{\mu,\gamma}(t)=t^3+\gamma t^2+\mu t+\gamma=0,~ \widetilde{\Phi}_{\mu,c}(t)=t^3+c t^2+ t+\mu c=0.\label{eq6:cubeqeven}
 \end{align}

\subsection{Incidence submatrices}
\begin{lemma}\label{lem6:Qinf0&tang}
Let $q\ge8$ be even.
The points $R_{\mu,\infty}$ and  $R_{\mu,0}$ lie on the tangents $\TT_1$ and $\TT_{\sqrt{\mu}}$, respectively.  No other point of $\ell_{\mu}$  is a $\Tr$-point, i.e. for all orbits generated by the lines $\ell_{\mu}$ we have $\Pb_{\Tr}=2$.
 \end{lemma}

\begin{proof}
We use the equations of the corresponding tangents given in  \cite[Lemma 6.2]{DMP_OrbLineArX}, \cite[Lemma 5.2]{DMP_OrbLineMedit}.
Also, by \eqref{eq6:tangintersec}, the equation $\varpi(\ell_\mu,\TT_t)=0$ has exactly two solutions $t=1$ and $t=\sqrt{\mu}$.
By \eqref{eq5:mutInvarW}, $\ell_\mu$  does not intersect $\TT_\infty$.
Finally, we use Lemma~\ref{lem3:orb the same}(ii).
\end{proof}

\begin{lemma}\label{lem6:}
 Let $q$ be even. For the orbits $\Os_\mu$
the following holds.
  \begin{align}\label{eq6:solut&incid}
 & \Pb_{\Tr}=\Pi_{2_\C}=2,~\Pb_{1_\Gamma}=\Nb_{1}(\mu);~ \Pi_{\overline{1_\C}}=\widetilde{\Nb}_{1}(\mu).
      \end{align}
\end{lemma}

\begin{proof}
By Lemma \ref{lem5:cubic_equation}, for a fixed $\gamma$, if the equation $\Phi_{\mu,\gamma}(t)$ \eqref{eq5:cubEq} (see also \eqref{eq6:cubeqeven}) has exactly $m$ distinct solutions in $\F_q$ then the point $R_{\mu,\gamma}$ belongs to exactly $m$ distinct osculating planes.  Also we use \eqref{eq3:2C=t}, Lemmas \ref{lem5:1Cplanes}, \ref{lem6:Qinf0&tang}.
\end{proof}

\begin{theorem}Let $q$ be even.
For the submatrices of the point-line and plane-line incidence matrices with respect to the orbit $\Os_\mu$, the values $\Pb_{\pk}$,
$\Lb_{\pk}$ and  $\Pi_{\pi}$,  $\Lambda_{\pi}$ (see Notation \emph{\ref{notation_2}}), are as in Tables \emph{\ref{tab:even_point-l}} and \emph{\ref{tab:even_plane-l}}, respectively; see also  Proposition \emph{\ref{prop3:PiG=LambG=PbC=LbC=0}}. The sizes of orbits $\#\M_\pk$ and $\#\N_{\pi}$ are noted in the top of columns.
\end{theorem}

\begin{proof}
  For $\Os_\mu$, we obtain the values of $\Pb_\Tr$ and $\Pb_{1_\Gamma}$ using \eqref{eq6:solut&incid}. Then we calculate $\Pb_{3_\Gamma}$ by \eqref{eq3:Pi1Pi2Pi3c} and $\Pb_{0_\Gamma}$ by \eqref{eq3:Pi1Pi2Pi3d}. Finally, we obtain all $\Lb_\pk$ by \eqref{eq3:obtainLamb1}. This gives Table \ref{tab:even_point-l}.

 We obtain $\Pi_{2_\C}$ and $\Pi_{\overline{1_\C}}$ by \eqref{eq6:solut&incid}, calculate $\Pi_{3_\C}$ by \eqref{eq3:Pi1Pi2Pi3}, $\Pi_{0_\C}$by \eqref{eq3:Pi1Pi2Pi3b}, and all $\Lambda_\pi$ by \eqref{eq3:obtainLamb1}. This gives Table \ref{tab:even_plane-l}.
\end{proof}

\begin{table}[htbp]
\caption{Values $\Pb_{\pk}$ (the  number of $\pk$-points on a line from  $\Os_\mu$, top entry) and $\Lb_{\pk}$ (the number of lines from  $\Os_\mu$ through a $\pk$-point, bottom entry) for the point-line incidence submatrices regarding the orbit $\Os_{\mu}$; $q$ is even,
 $\#\Os_{\mu}=(q^3-q)/2$}
\centering
\begin{tabular}{cccccc@{}}\hline
$q^2+q$&$\frac{1}{6}(q^3-q)$&$ \frac{1}{2}(q^3-q)$&$\frac{1}{3}(q^3-q)$\\
$\Pb_\Tr$&$\Pb_{3_\Gamma}$&$\Pb_{1_\Gamma}$&$\Pb_{0_\Gamma}$\\
$\Lb_\Tr$&$\Lb_{3_\Gamma}$&$\Lb_{1_\Gamma}$&$\Lb_{0_\Gamma}$\\\hline

$2$&$ \frac{1}{3}(q-3-\Nb_{1}(\mu))$&$ \Nb_{1}(\mu)$&$ \frac{2}{3}(q-\Nb_{1}(\mu))$\\
$q-1$&$q-3-\Nb_{1}(\mu)$&$\Nb_{1}(\mu)$&$q-\Nb_{1}(\mu)$\\\hline
\end{tabular}
\label{tab:even_point-l}
\end{table}

\begin{table}[htbp]
\caption{Values $\Pi_{\pi}$ (the  number of $\pi$-planes through a line from $\Os_\mu$, top entry) and
$\Lambda_{\pi}$ (the number of lines from  $\Os_\mu$ in a $\pi$-plane, bottom entry) for the plane-line incidence submatrices regarding the orbit $\Os_{\mu}$; $q$ is even,
 $\#\Os_{\mu}=(q^3-q)/2$}
\centering
\begin{tabular}{cccc}\hline
$q^2+q$&$\frac{1}{6}(q^3-q)$&$ \frac{1}{2}(q^3-q)$&$\frac{1}{3}(q^3-q)$\\
$\Pi_{2_\C}$&$\Pi_{3_\C}$&$\Pi_{\overline{1_\C}}$&$\Pi_{0_\C}$\\
$\Lambda_{2_\C}$&$\Lambda_{3_\C}$&$\Lambda_{\overline{1_\C}}$&$\Lambda_{0_\C}$\\\hline
$2$&$ \frac{1}{3}(q-3-\widetilde{\Nb}_{1}(\mu))$&$ \widetilde{\Nb}_{1}(\mu)$&$ \frac{2}{3}(q-\widetilde{\Nb}_{1}(\mu))\vphantom{H^{H^{H^H}}}$\\
$q-1$&$q-3-\widetilde{\Nb}_{1}(\mu)$&$\widetilde{\Nb}_{1}(\mu)$&$q-\widetilde{\Nb}_{1}(\mu)$\\\hline
\end{tabular}
\label{tab:even_plane-l}
\end{table}

\subsection{Obtaining of  $\Nb_{1}(\mu)$ and  $\widetilde{\Nb}_{1}(\mu)$ for equations $\Phi_{\mu,\gamma}(t)$ and $ \widetilde{\Phi}_{\mu,c}(t)$ of \eqref{eq6:cubeqeven}}

By \cite[Section 1.8]{Hirs_PGFF} the invariants $\delta$ and $\widetilde{\delta}$ of  $\Phi_{\mu,\gamma}(t)$ and $ \widetilde{\Phi}_{\mu,c}(t)$ of \eqref{eq6:cubeqeven} are, respectively,
\begin{align}\label{eq6:invar}
 &  \delta_{\mu,\gamma}=\frac{\mu^3+\gamma^4}{\gamma^2(\mu+1)^2}+\frac{1}{\mu+1},
 ~\widetilde{\delta}_{\mu,c}=\frac{1+\mu c^4}{c^2(\mu+1)^2}+\frac{\mu}{\mu+1},~\gamma, c\in\F_q^*.
   \end{align}

\begin{lemma}\label{lem6:1solution}
  Let $q$ be even. Let $\mathrm{Tr}_2(a)$ be the absolute trace of an element $a$ of $\F_q$. Then
\begin{align}\label{eq6:W}
  &\Nb_{1}(\mu))=\#\{\gamma\;|\;\mathrm{Tr}_2(\delta_{\mu,\gamma})=1,~\gamma\in\F_q^*\};
  ~\widetilde{\Nb}_{1}(\mu)=\#\{c\;|\;\mathrm{Tr}_2(\widetilde{\delta}{}_{\mu,c})=1,~c\in\F_q^*\}.
\end{align}
\end{lemma}

\begin{proof}
We use \cite[Corollary 1.15(ii)]{Hirs_PGFF}.
\end{proof}

\begin{example}\label{examp7}
In Table \ref{tab:examp even}, for even $q$, we give the values of $\Nb_1(\mu)$ and $\widetilde{\Nb}_1(\mu)$, obtained by the system MAGMA. By $\N$ we denote the number of $\mu$ providing the given values.
\begin{table}[h]
\caption{Values of $\Nb_1(\mu)$ and $\widetilde{\Nb}_1(\mu)$; $\N$ is the number of cases; even $q$}
\centering
\begin{tabular}{cccc|ccc|ccc}\hline
$q$&$2\Nb_1(\mu)$&$2\widetilde{\Nb}_1(\mu)$&$\N$&$2\Nb_1(\mu)$&$2\widetilde{\Nb}_1(\mu)$&$\N$&$2\Nb_1(\mu)$&$2\widetilde{\Nb}_1(\mu)$&$\N$
$\vphantom{H^{H^{H^H}}}$\\\hline

$2^3$&$q+2$&$q-4$&3&$q-4$&$q+2$&3\\\hline

$2^5$&$q+8$&$q-10$&5&$q-10$&$q+8$&5\\
&$q+2$&$q-4$&10&$q-4$&$q+2$&10\\\hline

$2^7$&$q+20$&$q-22$&7&$q-22$&$q+20$&7\\
&$q+14$&$q-16$&14&$q-16$&$q+14$&14\\
&$q+8$&$q-10$&21&$q-10$&$q+8$&21\\
&$q+2$&$q-4$&21&$q-4$&$q+2$&21\\\hline

$2^4$&$q-8$&$q-8$&2&$q-2$&$q-2$&8&$q+4$&$q+4$&4\\\hline

$2^6$&$q-14$&$q-14$&6&$q-8$&$q-8$&18&$q-2$&$q-2$&12\\
&$q+4$&$q+4$&12&$q+10$&$q+10$&14&\\\hline

$2^8$&$q-32$&$q-32$&8&$q-26$&$q-26$&16&$q-20$&$q-20$&16\\
&$q-14$&$q-14$&48&$q-8$&$q-8$&20&$q-2$&$q-2$&16\\
&$q+4$&$q+4$&56&$q+10$&$q+10$&16&$q+16$&$q+16$&18\\
&$q+22$&$q+22$&32&$q+28$&$q+28$&8&\\\hline
\end{tabular}
\label{tab:examp even}
\end{table}
Table \ref{tab:examp even} allows us to do the following interesting observations:
\begin{itemize}
  \item  Let $q=2^{2m}$, i.e. $q\equiv1\pmod3$. Then $\Nb_1(\mu)=\widetilde{\Nb}_1(\mu)$. The values of $\Nb_1(\mu)$ form an arithmetic progression with difference 3 and the first term $q/2-\Delta $, $\Delta\equiv1\pmod3$.
  \item  Let $q=2^{2m+1}$, i.e. $q\equiv-1\pmod{3}$. Then the set $\{(\Nb_1(\mu),\widetilde{\Nb}_1(\mu))|\mu\in\F_q^*\setminus\{1\}\}$ is partitioned into pairs of the form $(\Nb_1(\mu),\widetilde{\Nb}_1(\mu))=(a,b)$ and $(\Nb_1(\mu),\widetilde{\Nb}_1(\mu))=(b,a)$ where $(a,b)$ forms the series $(q/2+1,q/2-2)$, $(q/2+4,q/2-5)$, $(q/2+7,q/2-8)$ ...
      $(q/2+A,q/2-A-1)$, $A=3\cdot2^{m-1}-2$. The number $\N$ of $\mu$ providing each pair is divided by $2m+1$.
\end{itemize}

\end{example}

\section{Point-line and plane-line incidence submatrices for orbits $\Os_\lambda$ for odd  $q\not\equiv0\pmod3$}\label{sec:odd qne0}
In this section  $q$ is odd and $q\not\equiv0\pmod3$. By \eqref{eq5:eqtang},
it can be shown
 \begin{align}\label{eq7:solution tang}
 \mathfrak{n}_{q}(\mu)=\#\left\{t|t=\pm\sqrt{\frac{1}{2}\left(3\mu-1\pm\sqrt{(\mu-1)(9\mu-1}\right)},~ t\in\F_q\right\}\in\{0,2,4\}.
  \end{align}
\subsection{Incidence submatrices}
\begin{lemma}
 Let $q$ be odd, $q\not\equiv0\pmod3$.
 For the orbit $\Os_\mu$ we have
    \begin{align}\label{eq7:solut&incid}
 &  \Pb_{\Tr}=\Pi_{2_\C}=\mathfrak{n}_q(\mu)\T{ for all }q, \mu;\db\\
 &\Pb_{1_\Gamma}=\Nb_{1}(\mu)+1\T{ if }q\equiv-1\pmod3,~\Pb_{1_\Gamma}=\Nb_{1}(\mu)\T{ if }q\equiv1\pmod3;\dbn\\
  &\Pi_{\overline{1_\C}}=\widetilde{\Nb}_{1}(\mu)\T{ if }\mu\T{ is a square in }\F_q,~\Pi_{\overline{1_\C}}=\widetilde{\Nb}_{1}(\mu)+1\T{ otherwise}.\notag
      \end{align}
   \end{lemma}

\begin{proof}
By Lemma \ref{lem5:cubic_equation}, for a fixed $\gamma\in\F_q$, if $\Phi_{\mu,\gamma}(t)$ \eqref{eq5:cubEq} has exactly one solution $t$ in $\F_q$ then the point $R_{\mu,\gamma}$ belongs to exactly one osculating plane. So, the set $\ell_{\mu}\setminus\{R_{\mu,\infty}\}$ contains $\Nb_1(\mu)$ points belonging to exactly one osculating plane.  Then, using Lemmas \ref{lem5:R&RWmuinf} and \ref{lem5:Tpoints} we obtain $\Pb_{1_\Gamma}$ and $\Pb_{\Tr}=\Pi_{2_\C}$. Finally, by Lemmas \ref{lem3:orb the same} and \ref{lem5:1Cplanes}, we obtain $\Pi_{\overline{1_\C}}$.
\end{proof}

\begin{theorem}
Let $q$ be odd, $q\not\equiv0\pmod3$.
For the submatrices of the point-line and plane-line incidence matrices with respect to the orbit $\Os_\mu$, the values $\Pb_{\pk}$,
$\Lb_{\pk}$  and  $\Pi_{\pi}$,  $\Lambda_{\pi}$ (see Notation \emph{\ref{notation_2}}), are as in Tables \emph{\ref{tab:qodd_point-l}} and \emph{\ref{tab:qodd_plane-l}}, respectively; see also  Proposition \emph{\ref{prop3:PiG=LambG=PbC=LbC=0}}. The sizes of orbits $\#\M_\pk$ and $\#\N_{\pi}$ are noted in the top of columns.
\end{theorem}

\begin{proof}
  For $\Os_\mu$, we obtain the values of $\Pb_\Tr$ and $\Pb_{1_\Gamma}$ using \eqref{eq7:solut&incid}. Then we calculate $\Pb_{3_\Gamma}$ by \eqref{eq3:Pi1Pi2Pi3c} and $\Pb_{0_\Gamma}$ by \eqref{eq3:Pi1Pi2Pi3d}. Finally, we obtain all $\Lb_\pk$ by \eqref{eq3:obtainLamb1}. This gives Table \ref{tab:qodd_point-l}.

 Then we obtain  $\Pi_{2_\C}$, $\Pi_{\overline{1_\C}}$ by \eqref{eq7:solut&incid}, calculate $\Pi_{3_\C}$ by \eqref{eq3:Pi1Pi2Pi3}, $\Pi_{0_\C}$by \eqref{eq3:Pi1Pi2Pi3b}, and all $\Lambda_\pi$ by \eqref{eq3:obtainLamb1}. This gives Table \ref{tab:qodd_plane-l}.
\end{proof}

\begin{table}[h]
\caption{Values $\Pb_{\pk}$ (the  number of $\pk$-points on a line from  $\Os_\mu$, top entry) and $\Lb_{\pk}$ (the number of lines from  $\Os_\mu$ through a $\pk$-point, bottom entry) for the point-line incidence submatrices regarding the orbit $\Os_{\mu}$; odd $q\equiv\xi\pmod3$, $A_q(\mu)\triangleq q-2\nk_{q}(\mu)-\Nb_1(\mu)$,
$B_q(\mu)\triangleq 2 q-\nk_{q}(\mu)-2\Nb_1(\mu)$; $\boldsymbol{\ominus}$ means ``$\Upsilon_{q,\mu}$ \eqref{eq2:Upsilon} does not hold",
$\boldsymbol{\oplus}$ means ``$\Upsilon_{q,\mu}$  holds"}
\centering
\begin{tabular}{ccccccc}\hline
$\#\M_\pk\rightarrow$ &$q^2+q$&$\frac{1}{6}(q^3-q)$&$ \frac{1}{2}(q^3-q)$&$\frac{1}{3}(q^3-q)$\\
$\xi$&&&&\\
$\eta(\mu)$&$\Pb_\Tr$&$\Pb_{3_\Gamma}$&$\Pb_{1_\Gamma}$&$\Pb_{0_\Gamma}$\\
$\#\Os_{\mu}$&$\Lb_\Tr$&$\Lb_{3_\Gamma}$&$\Lb_{1_\Gamma}$&$\Lb_{0_\Gamma}$\\\hline

$\xi=-1$&&&&\\
 $\eta(\mu)=-1$&$\nk_{q}(\mu)$&$\frac{1}{3}A_q(\mu)$&$ \Nb_1(\mu)+1$&$\frac{1}{3}B_q(\mu)$\\
$\frac{1}{2}(q^3-q)$&$\frac{1}{2}(q-1)\nk_{q}(\mu)$&$A_q(\mu)$&$ \Nb_1(\mu)+1$&$\frac{1}{2}B_q(\mu)$\\

$\xi=-1$&&&&\\
$\eta(\mu)=1$&$\nk_{q}(\mu)$&$\frac{1}{3}A_q(\mu)$&$ \Nb_1(\mu)+1$&$\frac{1}{3}B_q(\mu)$\\
$\frac{1}{4}(q^3-q)$&$\frac{1}{4}(q-1)\nk_{q}(\mu)$&$\frac{1}{2}A_q(\mu)$&$ \frac{1}{2}(\Nb_1(\mu)+1)$&$\frac{1}{4}B_q(\mu)$\\

$\xi=1$&&&&\\
$\eta(\mu)=-1$&$\nk_{q}(\mu)$&$\frac{1}{3}(A_q(\mu)+1)$&$ \Nb_1(\mu)$&$\frac{1}{3}(B_q(\mu)+1)$\\
$\frac{1}{2}(q^3-q)$&$\frac{1}{2}(q-1)\nk_{q}(\mu)$&$A_q(\mu)+1$&$ \Nb_1(\mu)$&$\frac{1}{2}(B_q(\mu)+1)$\\

$\xi=1$&&&&\\
$\eta(\mu)=1,\boldsymbol{\ominus}$&$\nk_{q}(\mu)$&$\frac{1}{3}(A_q(\mu)+1)$&$ \Nb_1(\mu)$&$\frac{1}{3}(B_q(\mu)+1)$\\
$\frac{1}{4}(q^3-q)$&$\frac{1}{4}(q-1)\nk_{q}(\mu)$&$\frac{1}{2}(A_q(\mu)+1)$&$ \frac{1}{2}\Nb_1(\mu)$&$\frac{1}{4}(B_q(\mu)+1)$\\

$\xi=1$&&&&\\
$\eta(\mu)=1,\boldsymbol{\oplus}$&$\nk_{q}(\mu)$&$\frac{1}{3}(A_q(\mu)+1)$&$ \Nb_1(\mu)$&$\frac{1}{3}(B_q(\mu)+1)$\\
$\frac{1}{12}(q^3-q)$&$\frac{1}{12}(q-1)\nk_{q}(\mu)$&$\frac{1}{6}(A_q(\mu)+1)$&$ \frac{1}{6}\Nb_1(\mu)$&$\frac{1}{12}(B_q(\mu)+1)$\\\hline
\end{tabular}
\label{tab:qodd_point-l}
\end{table}

\begin{table}[h]
\caption{Values $\Pi_{\pi}$ (the  number of $\pi$-planes through a line from $\Os_\mu$, top entry) and
$\Lambda_{\pi}$ (the number of lines from  $\Os_\mu$ in a $\pi$-plane, bottom entry) for the plane-line incidence submatrices regarding the orbit $\Os_{\mu}$;
  odd $q\not\equiv0\pmod3$, $\widetilde{A}_q(\mu)\triangleq q-2\nk_{q}(\mu)-\widetilde{\Nb}_{1}(\mu)$,
$\widetilde{B}_q(\mu)\triangleq 2 q-\nk_{q}(\mu)-2\widetilde{\Nb}_{1}(\mu)$; $\boldsymbol{\ominus}$ means ``$\Upsilon_{q,\mu}$  \eqref{eq2:Upsilon} does not hold",
$\boldsymbol{\oplus}$ means ``$\Upsilon_{q,\mu}$  holds"}
\centering
\begin{tabular}{cccccc}\hline
$\#\N_\pi\rightarrow$&$q^2+q$&$\frac{1}{6}(q^3-q)$&$ \frac{1}{2}(q^3-q)$&$\frac{1}{3}(q^3-q)$\\
$\eta(\mu)$&$\Pi_{2_\C}$&$\Pi_{3_\C}$&$\Pi_{\overline{1_\C}}$&$\Pi_{0_\C}$\\
$\#\Os_{\mu}$&$\Lambda_{2_\C}$&$\Lambda_{3_\C}$&$\Lambda_{\overline{1_\C}}$&$\Lambda_{0_\C}$\\\hline

 $\eta(\mu)=-1$&$\nk_{q}(\mu)\vphantom{H^{H^{H^H}}}$&$\frac{1}{3}\widetilde{A}_q(\mu)$&$ \widetilde{\Nb}_{1}(\mu)+1$&$\frac{1}{3}\widetilde{B}_q(\mu)$\\
$\frac{1}{2}(q^3-q)$&$\frac{1}{2}(q-1)\nk_{q}(\mu)$&$\widetilde{A}_q(\mu)$&$ \widetilde{\Nb}_{1}(\mu)+1$&$\frac{1}{2}\widetilde{B}_q(\mu)$\\

$\eta(\mu)=1,\boldsymbol{\ominus}\vphantom{H^{H^{H^H}}}$&$\nk_{q}(\mu)$&$\frac{1}{3}(\widetilde{A}_q(\mu)+1)$&$ \widetilde{\Nb}_{1}(\mu)$&$\frac{1}{3}(\widetilde{B}_q(\mu)+1)$\\
$\frac{1}{4}(q^3-q)$&$\frac{1}{4}(q-1)\nk_{q}(\mu)$&$\frac{1}{2}(\widetilde{A}_q(\mu)+1)$&$ \frac{1}{2}\widetilde{\Nb}_{1}(\mu)$&$\frac{1}{4}(\widetilde{B}_q(\mu)+1)$\\

$\eta(\mu)=1,\boldsymbol{\oplus}\vphantom{H^{H^{H^H}}}$&$\nk_{q}(\mu)$&$\frac{1}{3}(\widetilde{A}_q(\mu)+1)$&$ \widetilde{\Nb}_{1}(\mu)$&$\frac{1}{3}(\widetilde{B}_q(\mu)+1)$\\
$\frac{1}{12}(q^3-q)$&$\frac{1}{12}(q-1)\nk_{q}(\mu)$&$\frac{1}{6}(\widetilde{A}_q(\mu)+1)$&$ \frac{1}{6}\widetilde{\Nb}_{1}(\mu)$&$\frac{1}{12}(\widetilde{B}_q(\mu)+1)$\\\hline
\end{tabular}
\label{tab:qodd_plane-l}
\end{table}

\subsection{Obtaining of  $\Nb_{1}(\mu)$ and $\widetilde{\Nb}_{1}(\mu)$ for equations $\Phi_{\mu,\gamma}(t)$ \eqref{eq5:cubEq} and  $\widetilde{\Phi}_{\mu,c}(t)$ \eqref{eq5:cub eq}}
For $\Phi_{\mu,\gamma}(t)$ \eqref{eq5:cubEq}, the discriminant $\Delta_{\mu,\gamma}$ and the  coefficients $A_i$ of the Hessian $H(T)=A_0T^2+A_1T+A_2$ , considered in \cite[ Section 1.8, equation (1.14), Lemma 1.18, Theorem 1.28, Proof of Lemma 1.32]{Hirs_PGFF},  are as follows:
\begin{align}\label{eq5:discrim}
 & \Delta_{\mu,\gamma}=27(3\gamma^2\mu^2-4\mu^3-4\gamma^4-\gamma^2+6\mu\gamma^2);\db\\
 &A_0=9(\mu-\gamma^2),~A_1=9\gamma(\mu-1),~A_1^2-4A_0A_2=-3\Delta_{\mu,\gamma}.\notag
   \end{align}
 For $\widetilde{\Phi}_{\mu,c}(t)$ \eqref{eq5:cub eq},  the corresponding objects $\widetilde{\Delta}_{\mu,c}$ and $\widetilde{A}_i$ have the form
 \begin{align}
 &\widetilde{\Delta}_{\mu,c}=c^2+4+4\mu c^4-27\mu^2 c^2+18\mu c^2\label{eq5:discrimW};\db\\
 &\widetilde{A}_0=-(3+c^2 ),~\widetilde{A}_1=c (1-9\mu);~\widetilde{A}_1^2-4\widetilde{A}_0\widetilde{A}_2=-3\widetilde{\Delta}_{\mu,c}.\notag
\end{align}

   \begin{lemma}
       Let $q$ be odd. Let $q\equiv\xi\pmod3$, $\xi\ne0$. Let the quadratic character $\eta$ be as in Section \emph{\ref{sec:incidL}}. Then the number of $\gamma\in\F_q$ and $c\in\F_q^*$ such that the equations $\Phi_{\mu,\gamma}(t)$ \eqref{eq5:cubEq} and $\widetilde{\Phi}_{\mu,c}(t)$ \eqref{eq5:cub eq} have exactly one solution  in $\F_q$ is:
  \begin{align}
   &\label{eq7:one solution}
   \Nb_1(\mu)= \#\{\gamma\,|\,\gamma\in\F_q,~\eta(-3\Delta_{\mu,\gamma})=-\xi\},\db\\
   &   \widetilde{\Nb}_1(\mu)= \#\{c\,|\,c\in\F_q^*,~\eta(-3\widetilde{\Delta}_{\mu,c})=-\xi\}.\label{eq7:one solutionW}
  \end{align}
\end{lemma}

\begin{proof}
We take $\Delta_{\mu,\gamma}$ and $A_i$ from \eqref{eq5:discrim}. By \cite[Corollary 1.30]{Hirs_PGFF}, for $\Delta_{\mu,\gamma}=0$, $\Phi_{\mu,\gamma}(t)$ has one root in $\F_q$ if all $A_i=0$. But $A_0\ne0$ if $\gamma=0$ and $A_1\ne0$ if $\gamma\ne0$. By \cite[Theorem 1,34, Table 1.3]{Hirs_PGFF}, for $\Delta_{\mu,\gamma}\ne0$, $\Phi_{\mu,\gamma}(t)$ has exactly one root  in $\F_q$, if $H(T)$ has 2 or 0 roots in $\F_q$ according to $q\equiv-1\pmod3$ or $q\equiv1\pmod3$, respectively. By \eqref{eq5:discrim}, $H(T)$ has 2 roots in $\F_q$ (resp. 0 roots) if and only if $-3\Delta_{\mu,\gamma}$ is a non-zero square (resp. non-square). This proves \eqref{eq7:one solution}.

The assertion \eqref{eq7:one solutionW} for $\widetilde{\Nb}_1(\mu)$ can be proved similarly.
\end{proof}

\begin{example}\label{examp-7}
In Table \ref{tab:examp odd}, for odd $q\not\equiv0\pmod3$, we give the values of $\mathfrak{n}_q(\mu)$, $\Nb_1(\mu)$, and $\widetilde{\Nb}_1(\mu)$, obtained by the system MAGMA. By $\N$ we denote the number of $\mu$ providing the given values.
\begin{table}[htbp]
\caption{Values of $\mathfrak{n}_q(\mu)$, $\Nb_1(\mu)$, and $\widetilde{\Nb}_1(\mu)$; $\N$ is the number of cases; odd $q\equiv\xi\pmod3$}
\centering
\begin{tabular}{cccc|ccc|ccc}\hline
$q$&$\xi$&$\eta(\mu)$&$\mathfrak{n}_q(\mu)$&$2\Nb_1(\mu)$&$2\widetilde{\Nb}_1(\mu)$&$\N$
&$2\Nb_1(\mu)$&$2\widetilde{\Nb}_1(\mu)$&$\N\vphantom{H^{H^{H^H}}}$\\\hline
5&-1&-1&0&$q-1$&$q-1$&2\\\hline
11&-1&-1&0&$q-7$&$q+5$&1&$q+5$&$q-7$&1\\
&&-1&2&$q-3$&$q-3$&3&&&\\
&&1&0&$q-1$&$q+1$&3&&\\\hline
17&-1&-1&0&$q-1$&$q-1$&4\\
&&-1&2&$q-9$&$q+3$&2&$q+3$&$q-9$&2\\
&&1&0&$q-7$&$q+7$&3&$q+5$&$q-5$&2\\
&&1&4&$q+1$&$q-9$&1\\\hline
23&-1&-1&0&$q-7$&$q+5$&4&$q+5$&$q-7$&4\\
&&-1&2&$q-3$&$q-3$&3\\
&&1&0&$q-1$&$q+1$&6\\
&&1&4&$q-5$&$q-3$&3\\\hline
29&-1&-1&0&$q-1$&$q-1$&6\\
&&-1&2&$q-9$&$q+3$&4&$q+3$&$q-9$&4\\
&&1&0&$q-7$&$q+7$&4&$q+5$&$q-5$&6\\
&&1&4&$q-11$&$q+3$&2\\\hline
7&1&-1&0&$q+3$&$q+1$&2\\
&&-1&2&$q-5$&0&1\\
&&1&0&$q-3$&$q-3$&1\\\hline
13&1&-1&0&$q-3$&$q-5$&4\\
&&-1&2&$q+1$&$q-1$&2\\
&&1&0&$q+3$&$q+3$&4\\\hline
19&1&-1&0&$q+3$&$q+1$&6\\
&&-1&2&$q-5$&$q-7$&2&$q+7$&$q+5$&1\\
&&1&0&$q-3$&$q-3$&4&$q+9$&$q+9$&1\\
&&1&4&$q-7$&$q-7$&2\\\hline
25&1&-1&0&$q-3$&$q-5$&4&$q+9$&$q+7$&4\\
&&-1&2&$q+1$&$q-1$&4\\
&&1&0&$q-9$&$q-9$&4&$q+3$&$q+3$&4\\
&&1&4&$q-1$&$q-1$&2\\\hline

31&1&-1&0&$q+3$&$q+1$&4&$q-9$&$q-11$&2\\
&&-1&2&$q-5$&$q-7$&7&$q+7$&$q+5$&2\\
&&1&0&$q-3$&$q-3$&7&$q+9$&$q+9$&4\\
&&1&4&$q+5$&$q+5$&2\\\hline
\end{tabular}
\label{tab:examp odd}
\end{table}
Table \ref{tab:examp odd} allows us to do the following interesting observations:
\begin{itemize}
  \item $\mathfrak{n}_q(\mu)\in\{0,2\}$ if $\eta(\mu)=-1$; $\mathfrak{n}_q(\mu)\in\{0,4\}$ if $\eta(\mu)=1$.
  \item  $\Nb_1(\mu)-\widetilde{\Nb}_1(\mu)=1$ if  $q\equiv1\pmod3$ and $\eta(\mu)=-1$.
  \item  $\Nb_1(\mu)=\widetilde{\Nb}_1(\mu)$ if  $q\equiv1\pmod3$ and $\eta(\mu)=1$.
  \item  Let $q\equiv-1\pmod{3}$, $\eta(\mu)=-1$. Then there are cases $\Nb_1(\mu)=\widetilde{\Nb}_1(\mu)$ and also pairs of the form $(\Nb_1(\mu),\widetilde{\Nb}_1(\mu))=(a,b)$ and $(\Nb_1(\mu),\widetilde{\Nb}_1(\mu))=(b,a)$.
  \item For $q'\equiv q''\pmod{12}$, relations between $\Nb_1(\mu)$ and $\widetilde{\Nb}_1(\mu)$ are similar.
\end{itemize}

\end{example}

\section{Point-line and plane-line incidence submatrices for orbits $\Os_\lambda$,  $q\equiv0\pmod3$}\label{sec:q=0mod3}
In this section $q\equiv0\pmod3$. This implies the natural simplification of some elements of Sections \ref{subsubsece:lmu} and \ref{sec:lmuW&intersec}. For $q\equiv0\pmod3$ we have $\mu\in\F_q^*\setminus\{1\}$ and
\begin{align}\label{eq8:sect5forq=0}
    & \varpi(\ell_\mu,\TT_t)=  t^4+ t^2+\mu,t\in\F_q;~\varpi(\ell_\mu,\TT_t)=0\T{ if }t=\pm\sqrt{1\pm\sqrt{1-\mu}};\db\\
    &\mathfrak{n}_q(\mu)=\#\left\{t|t=\pm\sqrt{1\pm\sqrt{1-\mu}},~t\in\F_q\right\}\in\{0,2,4\}.\notag
 \end{align}

 \subsection{Incidence submatrices}

\begin{lemma}\label{lem8:useful}
Let $q\equiv0\pmod3$. For an orbit $\Os_\mu$, \eqref{eq3:Pi1Pi2Pi3} and \eqref{eq3:Pi1Pi2Pi3b} hold. Also, we have
\begin{align}\label{eq8:q+qG}
&\textbf{\emph{(i)}}~~ \Pb_{(q+1)_\Gamma}=0, ~\Pb_\TO=\mathfrak{n}_{q}(\mu);\db\\
&\textbf{\emph{(ii)}}~\Pi_{\overline{1_\C}}=\Pb_\IC;\db\\
&\textbf{\emph{(iii)}}~\Pi_{\overline{1_\C}}=\widetilde{\Nb}_1(\mu) \T{ if }\mu\T{ is a square in }\F_q,~
\Pi_{\overline{1_\C}}=\widetilde{\Nb}_1(\mu)+1 \T{ otherwise}.
\end{align}
\end{lemma}

\begin{proof}
\begin{description}
  \item[(i)]
  An $\EnG$-line does not lie in any osculating plane, therefore it does not intersect the axis of the pencil of the osculating planes. See also Proposition \ref{prop3:PiG=LambG=PbC=LbC=0} and Lemmas \ref{lem3:tang&2cplane},
  \ref{lem5:Tpoints}.

  \item[(ii)]  The assertion follows from \cite[Table 1, Theorem 3.3(vi)]{DMP_PlLineIncJG} and \cite[Remark 3, (4.13)]{DMP_PointLineInc}.

  \item[(iii)] The assertion follows from Lemmas \ref{lem3:orb the same}, \ref{lem5:1Cplanes}.
 \qedhere
\end{description}
\end{proof}

\begin{theorem}
Let $q\equiv0\pmod3$.
For the submatrices of the point-line and plane-line incidence matrices with respect to the orbit $\Os_\mu$, the values $\Pb_{\pk}$,
$\Lb_{\pk}$ and  $\Pi_{\pi}$,  $\Lambda_{\pi}$ (see Notation~\emph{\ref{notation_2}}), are as in Tables \emph{\ref{tab:q=0mod3_point-l}} and \emph{\ref{tab:q=0mod3_plane-l}}, respectively; see also  Proposition \emph{\ref{prop3:PiG=LambG=PbC=LbC=0}} and Lemma~\emph{\ref{lem8:useful}}. The sizes of orbits $\#\M_\pk$ and $\#\N_{\pi(\pk)}$ are noted in the top of columns.
\end{theorem}

\begin{proof}
  For $\Os_\mu$, we obtain the values of $\Pb_\TO$, $\Pi_{2_\C}$, $\Pi_{\overline{1_\C}}$, and  $\Pb_\IC$, using Proposition \emph{\ref{prop3:PiG=LambG=PbC=LbC=0}} and Lemma \ref{lem8:useful}. Then we calculate $\Pb_\RC=q+1-\Pb_\TO-\Pb_\IC$, $\Pi_{3_\C}$ by \eqref{eq3:Pi1Pi2Pi3}, and $\Pi_{0_\C}$ by \eqref{eq3:Pi1Pi2Pi3b}. Finally, we obtain all $\Lb_\pk$ and $\Lambda_\pi$ by \eqref{eq3:obtainLamb1}.
\end{proof}

\begin{table}[h]
\caption{Values $\Pb_{\pk}$ (the  number of $\pk$-points on a line from  $\Os_\mu$, top entry) and $\Lb_{\pk}$ (the number of lines from  $\Os_\mu$ through a $\pk$-point, bottom entry) for the point-line incidence submatrices regarding the orbit $\Os_{\mu}$; $q\equiv0\pmod3$}
\centering
\begin{tabular}{cccc}\hline
$\#\M_\pk\rightarrow$ &$q^2-1$&$\frac{1}{2}(q^3-q)$&$ \frac{1}{2}(q^3-q)$\\
$\eta(\mu)$&$\Pb_\TO$&$\Pb_{\RC}$&$\Pb_{\IC}$\\
$\#\Os_{\mu}$&$\Lb_\TO$&$\Lb_{\RC}$&$\Lb_{\IC}$\\\hline

 $\eta(\mu)=-1$&$\nk_{q}(\mu)$&$q-\nk_{q}(\mu)-\widetilde{\Nb}_1(\mu)$&$ \widetilde{\Nb}_1(\mu)+1\vphantom{H^{H^{H^H}}}$\\
$\frac{1}{2}(q^3-q)$&$\frac{1}{2}q\nk_{q}(\mu)$&$q-\nk_{q}(\mu)-\widetilde{\Nb}_1(\mu)$&$\widetilde{\Nb}_1(\mu)+1$\\

$\eta(\mu)=1$&$\nk_{q}(\mu)$&$q+1-\nk_{q}(\mu)-\widetilde{\Nb}_1(\mu)$&$ \widetilde{\Nb}_1(\mu)\vphantom{H^{H^{H^H}}}$\\
$\frac{1}{4}(q^3-q)$&$\frac{1}{4}q\nk_{q}(\mu)$&$\frac{1}{2}(q+1-\nk_{q}(\mu)-\widetilde{\Nb}_1(\mu))$&$ \frac{1}{2} \widetilde{\Nb}_1(\mu)$\\\hline
\end{tabular}
\label{tab:q=0mod3_point-l}
\end{table}

\begin{table}[h]
\caption{Values $\Pi_{\pi}$ (the  number of $\pi$-planes through a line from $\Os_\mu$, top entry) and
$\Lambda_{\pi}$ (the number of lines from  $\Os_\mu$ in a $\pi$-plane, bottom entry) for the plane-line incidence submatrices regarding the orbit $\Os_{\mu}$;
 $q\equiv0\pmod3$, $\widetilde{A}_q(\mu)\triangleq q-2\nk_{q}(\mu)-\widetilde{\Nb}_1(\mu)$,
$\widetilde{B}_q(\mu)\triangleq 2 q-\nk_{q}(\mu)-2\widetilde{\Nb}_1(\mu)$}
\centering
\begin{tabular}{cccccc}\hline
$\#\N_\pi\rightarrow$&$q^2+q$&$\frac{1}{6}(q^3-q)$&$ \frac{1}{2}(q^3-q)$&$\frac{1}{3}(q^3-q)$\\
$\eta(\mu)$&$\Pi_{2_\C}$&$\Pi_{3_\C}$&$\Pi_{\overline{1_\C}}$&$\Pi_{0_\C}$\\
$\#\Os_{\mu}$&$\Lambda_{2_\C}$&$\Lambda_{3_\C}$&$\Lambda_{\overline{1_\C}}$&$\Lambda_{0_\C}$\\\hline

 $\eta(\mu)=-1$&$\nk_{q}(\mu)\vphantom{H^{H^{H^H}}}$&$\frac{1}{3}\widetilde{A}_q(\mu)$&$\widetilde{\Nb}_1(\mu)+1$&$\frac{1}{3}\widetilde{B}_q(\mu)$\\
$\frac{1}{2}(q^3-q)$&$\frac{1}{2}(q-1)\nk_{q}(\mu)$&$\widetilde{A}_q(\mu)$&$ \widetilde{\Nb}_1(\mu)+1$&$\frac{1}{2}\widetilde{B}_q(\mu)$\\

$\eta(\mu)=1\vphantom{H^{H^{H^H}}}$&$\nk_{q}(\mu)$&$\frac{1}{3}(\widetilde{A}_q(\mu)+1)$&$\widetilde{\Nb}_1(\mu)$&$\frac{1}{3}(\widetilde{B}_q(\mu)+1)$\\
$\frac{1}{4}(q^3-q)$&$\frac{1}{4}(q-1)\nk_{q}(\mu)$&$\frac{1}{2}(\widetilde{A}_q(\mu)+1)$&$ \frac{1}{2}\widetilde{\Nb}_1(\mu)$&$\frac{1}{4}(\widetilde{B}_q(\mu)+1)$\\\hline
\end{tabular}
\label{tab:q=0mod3_plane-l}
\end{table}

\subsection{Obtaining of  $\widetilde{\Nb}_{1}(\mu)$ for equation $\widetilde{\Phi}_{\mu,c}(t)$ \eqref{eq5:cub eq}}
\begin{lemma}
  Let the quadratic character $\eta$ be as in Section \emph{\ref{sec:incidL}}. For $\widetilde{\Nb}_1(\mu)$ we have
  \begin{align}\label{eq8:N1mu}
  \widetilde{\Nb}_1(\mu)=\#\left\{c\,|\,\eta\left(\frac{c^4}{\mu c^4+c^2+1}\right)=-1,~c\in\F_q^*,~\mu c^4+c^2+1\ne0\right\}.
  \end{align}
\end{lemma}

\begin{proof}
We use the results  of \cite[Corollary 1.23, Section 1.10, p. 25]{Hirs_PGFF}. We put $\alpha=-c^{-1}$, make the variable substitution $t=x+\alpha$  in  $\widetilde{\Phi}_{\mu,c}(t)$ and obtain $f(x)=x^3+cx^2+d=0$ where $d=\widetilde{\Phi}_{\mu,c}(\alpha)=-(\mu c^4+c^{2}+1)/c^3$. If $d=0 $ then  $\alpha$ is a double root of $\widetilde{\Phi}_{\mu,c}(t)$.
If $d\ne0$ then, by \cite[p. 25]{Hirs_PGFF},    $f(x)$ has exactly one root in $\F_q$  if $-c/d$ is a non-square in~$\F_q$.
\end{proof}

\begin{example}
In Table \ref{tab:examp q0}, for $q\equiv0\pmod3$, we give the values of $\mathfrak{n}_q(\mu)$ and $\widetilde{\Nb}_1(\mu)$ obtained by the system MAGMA. By $\N$ we denote the number of $\mu$ providing the given values. Similarly to Example \ref{examp7}, we  observe here that $\mathfrak{n}_q(\mu)\in\{0,2\}$ if $\eta(\mu)=-1$; $\mathfrak{n}_q(\mu)\in\{0,4\}$ if $\eta(\mu)=1$.
\begin{table}[htbp]
\caption{Values of $\mathfrak{n}_q(\mu)$, $\widetilde{\Nb}_1(\mu)$; $\N$ is the number of cases; $q\equiv0\pmod3$}
\centering
\begin{tabular}{ccc|cccccccc}\hline
$q$&$\eta(\mu)$&$\mathfrak{n}_q(\mu)$&$2\widetilde{\Nb}_1(\mu)$&$\N$&$2\widetilde{\Nb}_1(\mu)$&$\N$
&$2\widetilde{\Nb}_1(\mu)$&$\N\vphantom{H^{H^{H^H}}}$\\\hline
9&-1&0&$q+3$&2\\
&-1&2&$q-5$&2\\
&1&0&$q-1$&3&&\\\hline
27&-1&0&$q-3$&6&$q+9$&1\\
&-1&2&$q+1$&3&$q-11$&3\\
&1&0&$q+5$&6&$q-7$&3\\
&1&4&$q+1$&3\\\hline
81&-1&0&$q-9$&8&$q+3$&8&$q+15$&4\\
&-1&2&$q-17$&4&$q-5$&8&$q+7$&8\\
&1&0&$q-13$&4&$q-1$&8&$q+11$&8\\
&1&4&$q-17$&5&$q+7$&4\\\hline
\end{tabular}
\label{tab:examp q0}
\end{table}
\end{example}

\end{document}